\documentclass[oneside,fleqn,reqno,12pt,a4paper]{article}

\usepackage[a4paper,left=30mm,right=30mm,top=30mm,bottom=30mm,marginpar=20mm]{geometry}

 \usepackage[utf8]{inputenc}
 \usepackage[T1]{fontenc}  
\usepackage[english]{babel}

\usepackage{amsmath}
\usepackage{amssymb}
\usepackage{amsthm}
\usepackage{mathtools}
\usepackage{accents}
\usepackage{bbm}
\usepackage{bbold} 
\usepackage{color}
\usepackage[english=american]{csquotes}
\usepackage[final]{graphicx}
\usepackage{calc}
\usepackage{soul}
\usepackage{tikz}
\usepackage{graphicx}
\usepackage{stix}

\definecolor{luh-dark-blue}{rgb}{0.0, 0.313, 0.608}

\numberwithin{equation}{section}

\newtheoremstyle{thmlemcorr}{10pt}{10pt}{\itshape}{}{\bfseries}{.}{10pt}{{\thmname{#1}\thmnumber{ #2}\thmnote{ (#3)}}}
\newtheoremstyle{thmlemcorr*}{10pt}{10pt}{\itshape}{}{\bfseries}{.}\newline{{\thmname{#1}\thmnumber{ #2}\thmnote{ (#3)}}}
\newtheoremstyle{remexample}{10pt}{10pt}{}{}{\bfseries}{.}{10pt}{{\thmname{#1}\thmnumber{ #2}\thmnote{ (#3)}}}
\newtheoremstyle{ass}{10pt}{10pt}{}{}{\bfseries}{.}{10pt}{{\thmname{#1}\thmnumber{ A#2}\thmnote{ (#3)}}}

\theoremstyle{thmlemcorr}
\newtheorem{theorem}{Theorem}
\numberwithin{theorem}{section}
\newtheorem{lemma}[theorem]{Lemma}

\newtheorem{proposition}[theorem]{Proposition}

\newtheorem{definition}{Definition}
\numberwithin{definition}{section}
\newtheorem{theorem*}{Theorem}
\newtheorem{lemma*}[theorem]{Lemma}
\newtheorem{corollary*}[theorem]{Corollary}
\newtheorem{proposition*}[theorem]{Proposition}
\newtheorem{problem*}{Problem}
\newtheorem{conjecture*}{Conjecture}
\newtheorem{definition*}{Definition}
\newtheorem{assumption*}{Assumption}

\newtheorem{remark}[theorem]{Remark}

\newtheorem*{remark*}{Remark}

\newcommand{\Acal}{\mathcal{A}}

\newcommand{\Dcal}{\mathcal{D}}
\newcommand{\Ecal}{\mathcal{E}}

\newcommand{\Hcal}{\mathcal{H}}
\newcommand{\Ical}{\mathcal{I}}

\newcommand{\Lcal}{\mathcal{L}}
\newcommand{\Mcal}{\mathcal{M}}
\newcommand{\Ncal}{\mathcal{N}}

\newcommand{\Tcal}{\mathcal{T}}

\newcommand{\Wcal}{\mathcal{W}}
\newcommand{\Xcal}{\mathcal{X}}
\newcommand{\Ycal}{\mathcal{Y}}
\newcommand{\Zcal}{\mathcal{Z}}

\renewcommand{\Bbb}{\mathbb{B}}
\newcommand{\Cbb}{\mathbb{C}}
\newcommand{\Dbb}{\mathbb{D}}

\newcommand{\Hbb}{\mathbb{H}}

\newcommand{\Nbb}{\mathbb{N}}

\newcommand{\Rbb}{\mathbb{R}}

\newcommand{\Tbb}{\mathbb{T}}

\newcommand{\Zbb}{\mathbb{Z}}

\DeclareMathOperator{\dom}{dom}

\DeclareMathOperator{\ran}{ran}

\newcommand{\ee}{\mathrm{e}}
\newcommand{\ii}{\mathrm{i}}

\newcommand{\dd}{\;\mathrm{d}}
\newcommand{\DD}{\mathbb{D}}

\newcommand{\R}{\mathbb{R}}
\newcommand{\C}{\mathbb{C}}

\newcommand{\Z}{\mathbb{Z}}

\newcommand{\loc}{\mathrm{loc}}

\def\XXint#1#2#3{{\setbox0=\hbox{$#1{#2#3}{\int}$}
\vcentre{\hbox{$#2#3$}}\kern-.5\wd0}}

\renewcommand{\>}{\rangle}

\newcommand{\Df}{\mathfrak{D}}

\newcommand{\Hf}{\mathfrak{H}}
\newcommand{\hf}{\mathfrak{h}}

\newcommand{\Pf}{\mathfrak{P}}

\newcommand{\crm}{{\scriptscriptstyle \Cbb}}
\newcommand{\rrm}{{\scriptscriptstyle \Rbb}}

\newcommand{\al}{\alpha}
\newcommand{\ga}{\gamma} \newcommand{\Ga}{\Gamma}
\newcommand{\de}{\delta} 
\newcommand{\ep}{\epsilon}

\newcommand{\ka}{\kappa} 
\newcommand{\la}{\lambda} \newcommand{\La}{\Lambda}
\newcommand{\vphi}{\varphi}
  
\newcommand{\te}{\theta}
\renewcommand{\th}{\theta}
\newcommand{\Si}{\Sigma} \newcommand{\si}{\sigma}

\newcommand{\om}{\omega} \newcommand{\Om}{\Omega}

\newcommand{\one}{\mathbb{1}}
\newcommand{\zero}{\mathbb{0}}
\newcommand{\<}{\langle}
\newcommand{\pa}{\partial}

\newcommand{\wt}{\widetilde} 
\newcommand{\wh}{\widehat}

\newcommand{\rla}{\mathfrak{r}}
\newcommand{\vb}{\mathbf{v}}

\newcommand{\RR}{\mathbb{R}}
\newcommand{\Bo}{B}
\DeclareMathOperator{\Gr}{Gr}
\DeclareMathOperator{\sign}{sign}
\renewcommand{\Re}{\operatorname{Re}}
\renewcommand{\Im}{\operatorname{Im}}
\newcommand{\re}{\operatorname{Re}}

\newcommand{\Hyp}{\mathrm{Hyp}}
\newcommand{\dotL}{\accentset{\bullet}{L}}
\newcommand{\dotH}{\accentset{\bullet}{\Hcal}}

\newcommand{\Hsec}{\Hbb}

\newcommand{\rregH}{{\rla\text{-}\mathrm{reg}}}

\newcommand{\Gtwo}{E_2}

\usepackage{titlefoot}
\begin{document}

\title{Multi-parameter Hopf bifurcations\\ of rimming flows}

\author{\large Illya M. Karabash$^{\,1,*}$, Christina Lienstromberg$^{\,2}$, \\ and Juan J. L. Vel\'{a}zquez$^{\,1}$}

\date{}

\maketitle

\begin{abstract}
 In order to investigate the emergence of periodic oscillations of rimming flows, we study  analytically the stability of steady states for the model of (Benilov,  Kopteva, O'Brien, 2005), which describes the dynamics of a thin fluid film coating the inner wall of a rotating cylinder and includes effects of surface tension, gravity, and hydrostatic pressure. We apply multi-parameter perturbation methods to eigenvalues of Fréchet derivatives and prove  the transition of a pair of conjugate  eigenvalues from the stable to the unstable complex half-plane under appropriate variations of parameters. In order to prove  rigorously the corresponding branching of periodic solutions from critical equilibria, we extend the multi-parameter Hopf-bifurcation theory to the case of  infinite-dimensional dynamical systems.
\end{abstract}


 \footnotetext[1]{Institute of Applied Mathematics, University of Bonn, Endenicher Allee~60, 53115 Bonn, Germany. E-mails: karabash@iam.uni-bonn.de, velazquez@iam.uni-bonn.de}
 \footnotetext[2]{Institute for Analysis, Dynamics and Modeling, University of Stuttgart, Pfaffenwaldring 57, Stuttgart, Germany. E-mail:  Christina.Lienstromberg@mathematik.uni-stuttgart.de}
\unmarkedfntext{\!\!*Corresponding author. E-mail:karabash@iam.uni-bonn.de}

{\small
\medskip

\medskip

\noindent\textsc{Acknowledgement.} 
The authors were supported by the CRC 1060 `The mathematics of emergent effects' at the University of
Bonn, which is funded through the German Science Foundation (DFG, Deutsche Forschungsgemeinschaft), 
and by the Hausdorff Center for Mathematics funded by the the German Science Foundation  under Germany's Excellence Strategy – EXC-2047/1 – 390685813. 
I.Karabash is supported by the Heisenberg Programme (project 509859995) of the German Science Foundation.
}


\section{Introduction}



The problem of the dynamics of a liquid film on the surface a rotating horizontal cylinder  in a vertical gravity field has originated  from the studies of the behavior of
steam condensate in paper machine dryers \cite{ST11} and has a number of contemporary industrial applications including spin coating, spin casting, and rotational molding technologies \cite{LTH18}.
The considerable  interest  to this kind of systems is also connected with availability of a variety of experimental configurations \cite{ESR04}.  

A liquid film is called a rimming flow if the fluid is on the inside of a cylinder, and a coating flow if it is outside. Most of the mathematical studies of rimming and coating flows have been done for  the 2-dimensional geometry, where the flow is independent of the axial variable.
Works of Pukhnachev and Moffatt in 1970s  on the thin film limit for coating flows led to a great amount of research on the lubrication approximations derived from the governing Navier–Stokes equations, see \cite{BO05,BKO05,P05,K07,BBK08,LTH18} and references therein. The leading part of the resulting evolution equation is in the class of 4th order nonlinear PDEs, where a rigorous theory of weak solutions has been developing intensively starting from the works of Bernis and Friedman (in the context of rimming flows, see the references and discussions in \cite{CPT10,BCS12}). We would like to note that the lubrication
approximation has been used and studied extensively for a variety of fluid mechanics models  \cite{OO95,GO03,GKO08,GP08,LPV20,PV20}, and that, without the lubrication approximation, the exact Stokes equations for rimming flow and emergence of shocks were investigated numerically in \cite{BLO12}.

The main mathematical interest to the specifics of thin film models for rimming flows is motivated by a variety of nonlinear  phenomena  and  by several open questions \cite{CPT10,BCS12,CP14,LTH18}. 
For the cases where the leading order approximation has various additional differential terms (most important, the smoothing surface tension term), steady state solutions  were considered in \cite{K07,BBK08} for certain limiting regimes by numerical and approximation methods. Analytic studies of steady states (in particular, nonexistence theorems) go back to Pukhnachev and were continued  in \cite{CPT10}, where additionally the regularity of generalized weak solutions was investigated, as well as the short-time and long-time existence. The stability of steady states was intensively studied by approximation, numerical, and linearization methods, see  \cite{BOS03,BO05,BC11} and referenced therein.  
In the zero surface tension limit, the question on spectral effects behind unusual instability and 
ill-posedness effects raised in the paper \cite{BOS03} led to an intensive spectral analysis of nonselfadjoint differential operators appearing after taking of the Fréchet derivative at a steady state,
see \cite{D07,CKP09,CV09,DW10,BLM10} and references therein.
It was shown in \cite{CP14} that in this case the ill-posedness is also a feature of the original nonlinear equation, and so, is not an artefact of the linearization.

This paper investigates analytically the long-time behavior of solutions 
to the rimming-flow model  
\begin{align} \label{eq:PDE}
	& \pa_t h +
	\pa_\theta \left( \left(h - \frac{\ep_1 \cos \theta}{3}  h^3 \right)
	+
	b h^3 \left(\pa_\theta h + \pa_\theta^3 h\right)
	+
	\ep_2 \sin(\theta) h^3 \partial_\theta h \right)
	=
	0,
		\\
	& h(0,\theta) = h_0(\theta),
	\label{e:IC}
\end{align}
with $2\pi$-periodic in $\theta$ boundary conditions. Here $t > 0$, $\theta \in S^1$, and the circle $S^1 = \Rbb \ (\text{mod}\, 2\pi)$ is systematically identified with the interval $[0,2\pi)$.

 Namely, we study positive steady states, their stability, and periodic in $t$ solutions to
 \eqref{eq:PDE}.
Equation \eqref{eq:PDE} can be obtained by means of suitable rescaling of the lubrication approximation model of Benilov,  Kopteva, and O'Brien \cite{BKO05}, which describes the evolution of Newtonian fluid coating the inner surface of a rotating horizontal cylinder in the  thin layer limit and was designed to takes into account the effects of viscosity, gravity, surface tension, and hydrostatic pressure.

The function  $h=h(t,\theta) > 0$ is, roughly speaking, the nondimensional thickness of the thin film of fluid. The term $\pa_\th h$ in \eqref{eq:PDE} 
corresponds to the motion of the fluid due to the rotation of the cylinder. 
The term $\pa_\theta \left( \frac{\ep_1 \cos \theta}{3}  h^3 \right)$ 
 describes effects of the gravity. The term $\pa_\th\left( b h^3(\partial_\theta h + \partial_\theta^3 h)\right)$ accounts for the effect of the surface tension. Finally, the term $\pa_\th \left( \ep_2 \sin(\theta) h^3 \partial_\theta h \right)$ gives a correction corresponding to differences of the hydrostatic pressure at different parts of the fluid. 

Observe that \eqref{eq:PDE} is a quasi-linear problem of fourth order, which is degenerate para\-bolic in the sense that it lacks uniform parabolicity if $h$ vanishes at some points.
For small positive values of parameters  $\ep_1$ and $\ep_2$ the equation \eqref{eq:PDE}
 has a unique steady state $ H (\th; \ep_1,\ep_2)$ close to the constant solution $H_0 \equiv 1$, which the model has in the case $\ep_1=\ep_2 = 0$ (see Theorem \ref{thm:solution_ODE}).

The main points of the present paper are:
\begin{itemize}
\item We study analytically the stability properties of the stationary solutions $H (\cdot; \ep_1,\ep_2)$ to \eqref{eq:PDE}, see Theorems \ref{t:AHyp}--\ref{t:HB} and Sections 
\ref{ss:PertLa}--\ref{s:Stability}. These solutions  are obtained in Theorem \ref{thm:solution_ODE} by the application of the implicit function theorem around the explicit steady state $H_0(\theta;0,0)=\one$ corresponding to the unperturbed case $(\ep_1,\ep_2)=(0,0)$.

\item We prove rigorously that, for suitable choices of small $\ep_1$ and $\ep_2$, there exists a bifurcation of Poincaré--Andronov--Hopf type, i.e., an onset of a family of periodic solutions branching from the equilibrium $H (\cdot; \ep_1,\ep_2)$. The pairs of values of  $\ep_1$ and $\ep_2$, where the bifurcations take place, form an analytic curve in the parameter space, see Theorems \ref{t:HB} and \ref{t:n=2Det<0}.
\item In order to obtain the bifurcation result we employ the theory of multi-parameter Hopf-bifurcations, which is extended to infinite-dimensional dynamical systems in Section \ref{s:mpHopfB}. 
\item The analysis of positions of eigenvalues for the linearization of \eqref{eq:PDE} is performed with the use of the multi-parameter perturbation theory of eigenvalues, see Appendix \ref{a:A}.
\end{itemize}

 To finalize the passage to the analytical formulations and proofs, let us describe in more detail the background of the paper  \cite{BKO05}, which is our starting point, and the transition from the equations of \cite{BKO05} to the equation \eqref{eq:PDE}.

In \cite{BKO05} the dynamics of a rimming flow was studied for the model 
\begin{equation} \label{eq:PDE_unscaled}	
		\partial_{t} \tilde{h}+ 
		\partial_\theta\Bigl(
		\bigl(\tilde{h} - \frac{1}{3} \tilde{h}^3 \cos(\theta)\bigr)
		+
		\frac{\beta}{3} \tilde{h}^3 \left(\partial_\theta \tilde{h} + \partial_\theta^3 \tilde{h}\right)
		+
		\frac{\gamma}{3} \sin(\theta) \tilde{h}^3 \partial_\theta \tilde{h}
		\Bigr) = 0, 
\end{equation}
which is a simplified version of the nondimensional equation derived by Benilov and O'Brien in \cite{BO05} from the Navier--Stokes system by means of the  lubrication approximation. The model of \cite{BO05} contains more terms that give  for certain ranges of parameters additional information about the physics of the problem, specifically, about the fluid's inertia. In \cite{BKO05}, equation \eqref{eq:PDE_unscaled} was studied  as a version of the model of \cite{BO05} with the inertia terms ignored, and the stability and instability of the related steady states was examined by numerical computations of the spectrum for the linearized  equation.

The model of \cite{BO05,BKO05} describes the evolution of a thin layer of a Newtonian fluid coating the inner wall of a horizontal cylinder of radius $R > 0$. The cylinder rotates counterclockwise at fixed angular velocity $\omega > 0$ around a horizontal axis. The polar coordinates $(r,\theta)$ with the origin being located in the center of the cylinder are used to describe spatial positions.
It is assumed that the fluid has constant density $\rho > 0$, constant kinematic viscosity $\mu > 0$, and constant surface tension $\sigma > 0$.  We denote as $g > 0$ the gravity acceleration.

The nondimensionalization 
introduces the function 
$
\tilde{h} =  \frac{\wh h}{\gamma R} - \frac{\wh h^2}{2 \gamma R^2},
$
where $\wh h$ is the  height of the fluid layer and 
$
	\gamma = \sqrt{\frac{\mu \omega}{g R}}.
$
In the lubrication approximation limit, $\wh h$ is much smaller than $R$. So the function $\wt h$, which solves the equation \eqref{eq:PDE_unscaled},  is essentially the nondimensional height of the liquid film. 

The 
first two terms $(\tilde{h} - \tfrac{1}{3}\tilde{h}^3 \cos(\theta))$ in the (nondimensional) mass flux 
in \eqref{eq:PDE_unscaled} describe the leading order effects due to viscosity and gravity.  The term with the parameter $\gamma$ corresponds to hydrostatic pressure.
The term containing the parameter $	\beta = \frac{\sigma}{\rho g R^2} \sqrt{\frac{\mu \omega}{g R}}$ describes effects caused by surface tension.

Steady states of \eqref{eq:PDE_unscaled} are solutions $\wt{h}(t,\theta) = \wt{H}(\theta)$ to the ordinary differential equation
\begin{equation} \label{eq:ODE_unscaled}
	\bigl(\wt H - \tfrac{\cos \theta }{3} \wt H^3\bigr) 
	+ \tfrac{\beta}{3} \wt H^3 \bigl(\wt H' +  \wt  H'''\bigr) 
	+ \tfrac{\ga \sin \theta}{3} \wt H^3 \wt H' 
	= 
	\de,
	\quad
	\theta \in S^1,
\end{equation}
with periodic boundary conditions, where $\delta \in \R$ is the constant of integration and $f'$ means $\pa_\theta f$. Physically, $\delta$ is the nondimensional mass flux.

 
 In this paper our attention is concentrated on bifurcations of the Poincaré--Andronov--Hopf type
 (for the the main definitions of the bifurcation theory we refer to \cite{DvGLW12,K04,K95,L12}, 
 see also Section \ref{s:mpHopfB} below). 
 To avoid additional technical difficulties connected with the  degeneration of the leading term of the equation, 
 we consider a specific scaling, which will modify equation \eqref{eq:PDE_unscaled} to \eqref{eq:PDE}. 
Namely, let us assume that $\delta > 0$ is positive, but small. 
We are interested in steady states of \eqref{eq:PDE_unscaled} for which the modified height $\wt H$ of the fluid film is of the same order as the small parameter $\delta > 0$. 
In this case, we find that the term $\tfrac{1}{3} \cos(\theta) \wt{H}^3$, 
reflecting the gravitational effects, is of the same order of magnitude as $\delta^2 \wt{H}$.
Then, we choose the parameters $\beta$ and $\gamma$ in \eqref{eq:PDE_unscaled} such that: (a)
$
	\tfrac{\beta}{3} \wt{H}^3 (\partial_\theta \wt{H} + \partial_\theta^3 \wt{H}) $
is of the same order as $\wt{H}$, (b)
the terms \ $
	\tfrac{\gamma}{3} \sin(\theta) \wt{H}^3 \partial_\theta \wt{H} 
	$  \
	and \
	$\tfrac{1}{3} \cos(\theta) \wt{H}^3$ \ are of the same order of magnitude,
and (c) the terms due to surface tension and viscosity dominate those due to hydrostatic pressure and gravity. By rescaling equations \eqref{eq:ODE_unscaled} and \eqref{eq:PDE_unscaled} via 
	\ $h(t,\theta) = \frac{\wt{h}(t,\theta)}{\de}$,
	\
	$H(\theta) = \frac{\wt{H}(\theta)}{\de}$, 
	as well as
	\ $b = \frac{\beta \delta^3}{3}$,
	\ 
	$\ep_1 = \de^2$, \
	$\ep_2 = \frac{\ga \de^3}{3}$,
we obtain the ordinary differential equation for the rescaled steady states 
\begin{equation*} 
\left(H - \tfrac{\ep_1\cos \theta}{3}  H^3\right)
+
b H^3 \left(H' + H'''\right)
+
\ep_2 \sin(\theta) H^3 H'
= 
1, 
\quad \theta \in S^1,
\end{equation*}
and, respectively, we obtain the corresponding evolution equation \eqref{eq:PDE}, which is the main object of the present paper.
We assume $b > 0$ to be of order one, while $\ep_1, \ep_2$ 
are assumed to be small.

We are interested in the branching of periodic solutions at bifurcation points 
$\ep=(\ep_1,\ep_2)$ marking the passage from stable to unstable steady-states. 
Such bifurcation points $\ep=(\ep_1,\ep_2)$ are spectrally critical in the sense that a conjugate pair $\{\la^+,\la^-\}$ of nonzero eigenvalues of the linearized operator is situated on the pure imaginary line $\ii \Rbb = \{z \in \Cbb \ : \ \Re z = 0\}$ and crosses this line under specific small variations of $\ep$. This effect is usually associated with the variety of bifurcations of  Poincaré--Andronov--Hopf type that distinguish how the  eigenvalues $\la^\pm$ cross $\ii \Rbb$ (tangentially or nontangentially), positions of other eigenvalues, and various types of degeneracy, e.g., connected with higher multiplicities of eigenvalues.
In the context of multi-parameter bifurcations of finite-dimensional dynamical systems and retarded systems of ODEs we refer to \cite{C83,F86,I85,DK13} and to the detailed descriptions of various types of  the Poincaré--Andronov--Hopf bifurcations  in the monograph \cite{K95}.

The case, where the dynamical system is infinite-dimensional, brings a number of additional effects into the bifurcation theory, in particular, because of the necessity to trace the behavior of infinite number of eigenvalues of linearized operators. The study of bifurcations of the Poincaré--Andronov-Hopf type in this context seems to be initiated for the one-parameter case by Crandall and Rabinowitz \cite{CR77}.

Several types of Poincaré--Andronov-Hopf bifurcations for infinite-dimensio\-nal dynamical systems were described in monographs \cite{K04,DvGLW12,L12}. However, the settings of these monographs do not fit to specifics of the rimming-flow equations \eqref{eq:PDE_unscaled} and \eqref{eq:PDE}
by two reasons:
\begin{itemize}
\item The models of \cite{BO05,BKO05}, as well as equation \eqref{eq:PDE}, include several independent parameters, which makes it natural to consider them from the point of view of the multi-parameter bifurcation theory. Multi-parameter Hopf-bifurcations in the infinite-dimensional phase spaces are not fully understood yet (see the discussion in \cite{DK13}). 
\item In comparison with the standard spectral diagram for Poincaré--Andronov-Hopf bifurcations, the linearized operators for the dynamical systems \eqref{eq:PDE_unscaled} and \eqref{eq:PDE} have always an additional eigenvalue at zero, which corresponds to the mass conservation. This effect can be handled by the restriction of the dynamical system  to the invariant hyperplanes
\begin{equation} \label{e:HypC}
	\Hyp (C) = \{ h \in L^2(S^1): \ ( h | \one) = C\} \qquad \text{with $C > 0$, }
\end{equation}
where $( \cdot |\cdot)$ is the $L^2(S^1)$-inner product, and $\one$ is the constant function equal $1$ on the whole $S^1$.
This removes the zero eigenvalue, but brings into \eqref{eq:PDE} an additional parameter.
\end{itemize}

In Section \ref{s:mpHopfB}, we develop an abstract version of the multi-parameter Hopf-bifurcation theory, which deals with infinite--dimensional dynamical systems and is sufficient for the specifics of the rimming-flow model under consideration. We show that the rimming-flow model \eqref{eq:PDE} fits to a specific singular, but relatively simple,  case when two bifurcation curves cross each other nontangentially (see Sections \ref{ss:singular} and \ref{s:rfHcurves}). This provides then the analytic basis for the computation of the tangent vector of physically relevant branch of the Hopf-bifurcation locus in the the first quadrant $\{\ep_1>0,\ep_2>0\}$ of the $\ep$-parameter plane (see Theorem \ref{t:AHyp}).

\textbf{Notation.}
By $\Xcal$ and $\Ycal$ we denote general Banach spaces over $\Rbb$ or $\Cbb$. By $\Lcal(\Xcal,\Ycal)$ we denote the Banach space of bounded (linear) operators $T\colon\Xcal \to \Ycal$. In particular,  $\Lcal (\Xcal) := \Lcal(\Xcal,\Xcal)$. If $\Xcal \subset \Ycal$, we sometimes consider linear operators $T\colon\Xcal \to \Ycal$, as operators in the space $\Ycal$ with the \emph{domain} $\dom T = \Xcal$, see \cite{Kato}. The corresponding notation is $T\colon\Xcal \subset \Ycal \to \Ycal$.
By $\ker T = \{y \in \dom T \ : \ T y = 0\}$ the \emph{kernel} of a linear operator $T$ is denoted, and by 
$\ran T = \{Ty \ : \ y \in \dom T \}$ its \emph{range}.
By $I=I_\Xcal$  we denote the identity operator on $\Xcal$. 

Inner products in real or complex Hilbert spaces $\Xcal$ are denoted by  $(\cdot | \cdot)_{\Xcal}$.
Moreover, we use $\Bbb_{r} (v_0) = \Bbb_{r}(v_0;\Xcal) = \{v \in \Xcal \ :\ \| v - v_0 \|_\Xcal < r \}$ for an open ball of radius $r > 0$ in $\Xcal$. In particular,  $\Dbb_r (z_0) := \Bbb_{r}(z_0;\Cbb)$  is an open complex disc. A \emph{neighborhood $\Ncal^x$ of a point} $x \in \Xcal$ is an open set containing $x$. To distinguish several different neighborhoods of $x$, we often index them with the subscript, i.e., $\Ncal^x_0$, $\Ncal^x_1$, \dots.
By $\Nbb = \{1,2,\dots\}$ the set of natural numbers is denoted, while $\Nbb_0 :=  \{0\} \cup \Nbb$. 

The notation $\< \cdot | \cdot  \>_{\Rbb^n}$ stands for the standard inner product in $\Rbb^n$, while 
 \[
 ( f|g)= \tfrac1{2\pi} \int_0^{2\pi} f \overline{g} \dd \theta
 \]
  is an inner product in the complex  Hilbert space $L^2_\crm (S^1) = L^2 (S^1;\Cbb) $. By $\sign(\cdot)$ we denote the standard signum function.


\section{Main results of the paper}\label{s:MainRes}

The parameter $b > 0$ is fixed through the rest of the paper.
We start from the rigorous description of a 2-dimensional analytic surface of steady states of 
the rimming-flow equation 
\begin{align} \label{eq:PDE2}
	& \pa_t h +
	\pa_\theta \left( \left(h - \frac{\ep_1 \cos \theta}{3}  h^3 \right)
	+
	b h^3 \left(\pa_\theta h + \pa_\theta^3 h\right)
	+
	\ep_2 \sin(\theta) h^3 \partial_\theta h \right)
	=
	0,	
\end{align}
where $t > 0$ and $ \theta \in S^1$.
This equation is parametrized by  $\ep=(\ep_1,\ep_2)$ in a sufficiently small $\Rbb^2$-neighborhood of $\ep^0 = (0,0)$. As before, we identify the circle $S^1 = \Rbb \ (\text{mod}\, 2\pi)$ with the interval $[0,2\pi)$, i.e., solutions $h$ are $2\pi$-periodic in $\th$. (Recall that,  the physically relevant quadrant of  the parameter plane $\Rbb^2$   is 
$\{\ep_1\ge 0, \ep_2 \ge 0\}$.)

Steady-state solutions to \eqref{eq:PDE2} are the time independent solutions $h(t,\theta) = H(\theta) = H (\theta; \ep_1,\ep_2)$ to the ordinary differential equation
\begin{equation} \label{eq:ODE}
\left(H - \tfrac{\ep_1\cos \theta}{3}  H^3\right)
+
b H^3 \left(H' + H'''\right)
+
\ep_2 \sin(\theta) H^3 H'
= 
1, 
\quad \theta \in S^1,
\end{equation}
where $H'$ stands for $\pa_\th H$ and  $H^j  = \left( H (\theta;\ep)\right)^j$ denotes the  $j$-th power of $H (\theta;\ep)$. 
Note that in the case $\ep=(0,0)$, the constant function $H(\theta) \equiv \one$ is an explicit stationary solution to \eqref{eq:PDE2} satisfying \eqref{eq:ODE}.



The following functional spaces and norms are used. We identify functions in the (real) Lebesgue spaces $L^2(S^1)$ 
with $2\pi$-periodic functions in $L^{2,\text{loc}}(\Rbb)$. 
The real Hilbert space $L^2(S^1)$ and its complex version $L^2_\crm (S^1)$ are equipped with the inner product 
$( f|g)= \tfrac1{2\pi} \int_0^{2\pi} f \overline{g} \dd \theta$.
The induced norm is denoted by $\| f \|_{L^2} = \sqrt{(f|f)}$. With this choice of the inner product,
$\{ \ee^{\ii n \theta}\}_{n\in \Zbb}$ is an orthonormal basis in $L^2_\crm(S^1)$. 
By 
$\Hcal^{k}(S^{1}) = \Wcal^{k,2}(S^1) $ we denote the corresponding real Hilbertian Sobolev spaces; in particular, $\Hcal^{0}(S^{1})$ and $L^2(S^{1})$ coincide as linear spaces and have equivalent norms.
We can represent any $v\in \Hcal_\crm^{k}(S^{1})$ by its Fourier series
$ v(\theta)=\sum_{n\in\Z} v_n e^{\ii n\theta}$, where the Fourier coefficients $v_n \in \Cbb$ are such that 
$ \| v \|_{\Hcal^{k}} = \left( \sum_{n\in\Z} |v_n|^2(n^{2k}+1)  \right)^{1/2} < \infty$. 
Denoting 
\begin{gather*} \label{e:dotLDef}
\dotL^2_\crm(S^1):= \{u \in L^2_\crm (S^1)  :  \ (u|\one )=0 \} \quad \text{ and } \quad \dotL^2 (S^1):= L^2(S^1) \cap \dotL^2_\crm(S^1),
\end{gather*}
one sees that $\dotL^2_{(\crm)} (S^1)$ is a closed subspace of $L^2_{(\crm)}(S^1)$. 
More generally, for $ s \in \Nbb \cap \{0\}$, 
 the closed subspaces $\dotH^s_{(\crm)} (S^1) :=  \Hcal^s_{(\crm)}(S^1) \cap \dotL^2_\crm(S^1)$ of $\Hcal^s_{(\crm)}(S^1)$,  are real (respectively, complex) Hilbert spaces.

\begin{theorem}\label{thm:solution_ODE}
There exist constants $\de_0, \de_1>0$ such that the following statements hold for all $\ep=(\ep_1,\ep_2)$ in the neighborhood $\Bbb_{\de_0} (0;\Rbb^2)$:
\begin{itemize}
\item[(i)]  the steady state equation \eqref{eq:ODE} admits a unique solution $H(\cdot) = H(\cdot;\ep)$ satisfying 
 \ $\|H(\cdot;\ep) - \one\|_{\Hcal^4(S^1)} < \de_1 $;  
	\item[(ii)] the $\Hcal^4(S^1)$-valued map $\ep \mapsto H (\cdot;\ep)$ is real analytic in 
	$\ep$ in $\Bbb_{\de_0} (0;\Rbb^2)$;
	\item[(iii)]	$\min\limits_{\theta \in S^1} |H (\theta;\ep)| >0 $;
	\item[(iv)] the first terms of the Maclaurin  series of the $\Hcal^4(S^1)$-valued map $\ep \mapsto H (\cdot;\ep)$ are given by the asymptotic formula 
	\ $  H(\theta;\ep) = \one + \frac{\ep_1}{3} \cos(\theta) + O(|\ep|^2) $ \quad as $\ep\to 0$.
\end{itemize}
\end{theorem}

The proof is given in Section \ref{s:steady} with the use of the implicit function theorem for analytic maps applied around $\ep^0 = (0,0)$ and $H(\theta;\ep^0)=H_0 \equiv \one$. 

\begin{remark}
The space $\Hcal^4 (S^1)$ can be replaced in Theorem \ref{thm:solution_ODE} by any space $\Hcal^s (S^1)$ with $s \ge 3$, see Proposition \ref{p:CsolutionODE}.
\end{remark}

The quantity $(\ep_1^{1/2} h | \one) = (\de h | \one) = (\wt h | \one)$ corresponds to the conserved total mass in the evolution equation \eqref{eq:PDE_unscaled}.
Mathematically, the fact that hyperplanes $\Hyp (C) = \{ h \in L^2(S^1): \ ( h | \one) = C\}$ are invariant sets of \eqref{eq:PDE2} is easily seen in the following way. We write \eqref{eq:PDE2} as the following dynamical system in the Hilbert space $L^2(S^1)$:
\begin{gather} \label{eq:f}            
\pa_t h  = f(h;\ep)  \qquad \text{ with \ $ f (h;\ep) = -\pa_\th f_1 (h;\ep)$ \ and } \\
\text{ $ f_1 (h;\ep) =   \left(h - \frac{\ep_1 \cos \theta }{3}  h^3 \right)
        + b h^3 \left(\pa_\theta h + \pa_\theta^3 h\right) + \ep_2 \sin(\theta) h^3 \pa_\theta h $,} \label{eq:f1}         
\end{gather}
where $\ep$ is an $\Rbb^2$-parameter. 
One sees  that 
$   \pa_t(  h | \one ) = 
( f_1 (h;\ep) | \pa_\theta \one ) =0$. 
Hence, if $C= (h_0|\one)$ for the initial condition 
\begin{align} \label{e:h0}
h |_{t=0} = h_0 ,
\end{align}
then a solution $h$ to \eqref{eq:f}-\eqref{e:h0} satisfies 
$(h(t)|\one) = C$ for all $t \ge 0$. Here and below we perceive $f$ as a real analytic map from an (open) $\Hcal^4(S^1) \times \RR^2$-neighborhood of $(H_0,\ep^0) =  (\one; \ep^0)$ to  $L^2(S^1)$ and treat solutions $h$ to the initial value problem \eqref{eq:f}-\eqref{e:h0} as 
strict \ $C^1([0,t_0);L^2 (S^1)) \cap C([0,t_0) ;\Hcal^4(S^1))$-solutions (we use here the settings of \cite{L12}).

It is not difficult to see from the form of the map $ f$ that, for $\ep$ is a vicinity of $\ep^0=(0,0)$,  the Fréchet derivatives $D_w f (w_0;\ep)$ of $f$ with respect to (w.r.t.) $w$ at a steady state $w_0 (\cdot) = H(\cdot;\ep)$ has an eigenvalue at $0$.
 We need to remove this eigenvalue before we pass to the analysis of the stability and Hopf bifurcations. To this end, the dynamical system \eqref{eq:f} is considered as a family of dynamical systems in the invariant hyperplanes $\Hyp (C)$. 
For the purposes of the spectral analysis of Fréchet derivatives we also consider the complexified version of \eqref{eq:f} in the complex Hilbert space $L^2_\crm (S^1)$. 

Let us take an arbitrary fixed $\ep$ in the neighborhood $\Bbb_{\de_0} (0;\Rbb^2)$ described in Theorem \ref{thm:solution_ODE} and consider the corresponding steady state $H(\ep) = H (\cdot,\ep) $.
Then $H(\ep)  \in \Hyp (C(\ep))$ with $C(\ep)= \int_{S^1} H(\theta,\ep) \frac{\dd \theta}{2\pi}$.
Using the representation 
$
\Hyp (C(\ep)) = \{  H(\cdot;\ep) + v(\cdot) \ : \  v (\cdot) \in \dotL^2 (S^1)  \},
$
we write the restriction of system \eqref{eq:f} to the invariant hyperplane $\Hyp (C(\ep))$  as a dynamical system 
w.r.t. the state $u(t,\theta,\ep) := h(t,\theta,\ep)- H (\theta,\ep)$ in the complex Hilbert space 
$ \dotL^2_\crm(S^1) $   in the following way
\begin{gather} \label{e:PDEHyp2}
        \pa_t u =  F (u ; \ep) , \qquad \text{ where  $F (w ;\ep) = f (w+H(\ep);\ep)$}.             
\end{gather}
Then the identically zero function $\zero$ is a steady state of (\ref{e:PDEHyp2}) that corresponds to the steady state  
$H(\cdot,\ep)$ of the original equation (\ref{eq:PDE2}).

The complexification $F_\crm$ of the mapping $F$ can be defined by the same formula as $F$. However, for this purpose  $H (\ep)$ have to be understood as a holomorphic extension of the real analytic function $H (\cdot)$ from Theorem \ref{thm:solution_ODE}. Such a holomorphic extension to a certain $\Cbb^2$-neighborhood 
$\Bbb_{\de_2} (0;\Cbb^2)$ of $\ep^0=(0,0)$ with $\de_2 >0 $ can be obtained by means of the complex Maclaurin  series of $H (\cdot)$. 

The complexification of the Fréchet derivative $D_w F(\zero;\ep)$ is given by the operator 
\begin{equation} \label{e:Av}
	A(\zero;\ep) := D_w F_\crm(0;\ep) 
		  \in \Lcal\bigl(\dotH^4_\crm(S^1), \dotL^2_\crm(S^1)\bigr),
\end{equation}
 defined for $v\in \dotH^4_\crm(S^1)$ by $A(\zero;\ep)[v]  =  	- \pa_\theta Q_\ep[v] $, where
\begin{multline}
	Q_\ep [v] 
	=
	\bigl(v - \ep_1 \cos(\theta) H^2  v\bigr)  
	 +
	b\Bigl(3H^2  (\pa_\theta H  + \pa_\theta^3 H ) v + 
	H^3  (\pa_\theta v + \pa_\theta^3 v)\Bigr)  
	\\
		+
	\ep_2 \Bigl(3 \sin(\theta) H^2 v \pa_\theta H  + 
	\sin(\theta) H^3  \pa_\theta v\bigr).
	\label{e:A1}
\end{multline}

We call $A(\zero;\ep)$ the \emph{linearized operator} (corresponding to the 2-parameter $\ep$). To describe some of its spectral properties, the following notation and definitions of spectral theory are used (see \cite{Kato,RSIV}).
By $\DD_r (\zeta_0) = \{ \zeta \in \Cbb \ : \ |\zeta -\zeta_0|<r \}$ we denote a complex open disc of radius $r>0$ with the center at $\zeta_0 \in \Cbb$, and by $\Dbb_r^2 (\zeta_0) = (\Dbb_r (\zeta_0))^2$ the polydisc that is the direct product of two discs $\DD_r (\zeta_0)$.
For a linear operator $T\colon \dom T \subset \Ycal \to \Ycal$ in a complex Banach space $\Ycal$, we denote by $\dom T$ its domain (of definition), by $\sigma (T)$ its \emph{spectrum}, and by 
$\rho (T) = \Cbb \setminus \si (T)$ its \emph{resolvent set}. An eigenvalue $\la$ of $T$ is called \emph{isolated} if there exists a complex open neighborhood $\Dbb_r (\la)$ of $\la$ such that $\{\la\} = \sigma (T) \cap \Dbb_r (\la)$. An eigenvalue is called \emph{simple} if it has algebraic multiplicity $1$. 
Roughly speaking, the open left half-plane 
$\{\la \in \Cbb \ :\ \Re \la <0\}$ in the complex plane of the spectral parameter is associated with  the case of stability (see Theorem \ref{t:HB}).

The following theorem describes the behavior under variation of $\ep$ near  $\ep^0 = (0,0)$ of the pair of eigenvalues of the linearized operator which are closest to the purely imaginary axis $\ii \Rbb$.

\begin{theorem} \label{t:AHyp}
 For certain small enough $\de_3, \de_4,\de_5>0$, there exist two holomorphic functions
   $\la^\pm : \Dbb_{\de_3} (0) \times \Dbb_{\de_4} (0) \to \Cbb$
   and an increasing real analytic function 
  $\ep_1 \mapsto \Gtwo(\ep_1), \ \Gtwo  : (-\de_3,\de_3) \to \RR$, 
with the following properties:
\begin{itemize}
		\item[(i)]  For all $\ep \in (-\de_3,\de_3) \times (-\de_4,\de_4)$, the part $\{\la \in \si (A(\zero;\ep)) \ : \ - \de_5 < \Re \la \}$ of the spectrum of the linearized operator $A(\zero;\ep)$ 
		consists of exactly two simple isolated eigenvalues $\la^\pm (\ep)$ such that 
		\[
		\text{$\la^\pm (\ep^0) = \pm \ii$, \  $\Re \la^+ (\ep) = \Re \la^- (\ep)$, \ and \ $\Im \la^+ (\ep) = - \Im \la^- (\ep) >0$;}
		\]		
moreover, the operator $A(\zero;\ep)$ is a generator of quasi-bounded holomorphic semigroup (see  Remark \ref{r:sectorial}).
	\item[(ii)] For $\ep \in [0,\de_3) \times [0,\de_4)$,
\begin{align}
& \text{$\Re \la^\pm (\ep) = 0  $ \quad  is equivalent to \quad $\ep_2 = \Gtwo (\ep_1) $},\\
	& \text{$\Re \la^\pm (\ep) < 0 $ \quad  is equivalent to \quad $\  \ep_2 \  < \ \Gtwo (\ep_1)$,} \\
	& \text{$\Re \la^\pm (\ep) > 0  $ \quad is equivalent to \quad $\Gtwo (\ep_1) \  <  \ep_2 \    $}.
\end{align} 
	\item[(iii)] The leading terms of the Maclaurin series of $\Gtwo (\cdot)$ are given by 
	\begin{equation} \label{e:Ec2}
		\Gtwo (\ep_1) = \frac{3b}{2} \left(1 + \sqrt{1+4/b } \right) \ep_1 + O (\ep_1^2) 
		\qquad \text{as $\ep_1 \to 0$.}
	\end{equation} 
\end{itemize}
\end{theorem}

This theorem is proved in Section \ref{ss:PertLa}.

Controlling the rest of eigenvalues for small enough $\de_2$, 
we show that 
the graph $\Gr \Gtwo$ of the function $\Gtwo$,  $\Gr \Gtwo := \{(\ep_1,\Gtwo (\ep_1)) \ : \ \ep_1 \in (-\de_3,\de_3)\}$, 
separates in the first quadrant $\{\ep_1>0,\ep_2>0\}$ the regions of  instability and exponential stability  for the steady state $\zero$ of the reduced system \eqref{e:PDEHyp2}.
 That is, it separates the regions of instability and exponential stability of the steady state $H(\ep)$ of \eqref{eq:PDE2} if the dynamics considered in the corresponding hyperplane $\Hyp (C(\ep))$.
(Note that, because of the invariance of the hyperplanes $\Hyp (C(\ep))$, the steady states $H(\ep)$ obviously cannot be asymptotically stable in the sense of the original equation \eqref{eq:PDE2}).

It will be shown in Theorem \ref{t:HB}  that 
the graph $\Gr \Gtwo$ of the function $\Gtwo (\ep_1)$ is essentially a parametric bifurcation curve 
from the point of view of  multi-parameter Hopf-bifurcations (see Definition \ref{d:HbManifold} and Theorem \ref{t:n-parHb} below; for the basics of the theory in finite-dimensional settings see  \cite{K95}). 

The rigorous description of multi-parameter Hopf bifurcations is given in Section \ref{s:HBon Paths}, where we adapt to infinite-dimensional dynamical systems the approach suggested by Diekmann and Korvasov\'a  \cite{DK13} for 2-parameter Hopf bifurcations of delay differential equations. 
Namely, if the parameter $\ep$ moves along a $C^3$-path $\Ecal (s) $, $s \in [s_-,s_+]$,  that crosses a Hopf-bifurcation curve $\Gr \Gtwo$ nontangentially at a certain point $\ep = \Ecal (s_0) \in \{\ep_1>0,\ep_2>0\} $, then 1-parameter Hopf bifurcation  takes place at this point w.r.t. the new one-dimensional parameter $s$. That is, a family of periodic nonstationary solutions branches from the curve of steady states at $\ep = \Ecal (s_0)$. 
These stability and bifurcation results are summarized in the next theorem.

\begin{theorem} \label{t:HB}
Assume that $\de_3,\de_4>0$ and the function $\ep_1 \mapsto \Gtwo (\ep_1)$ are as in Theorem \ref{t:AHyp}.  Let us denote
$\Om_- := \{ \ep \in  [0,\de_3) \times [0,\de_4)  :  \ep_2 \  <  \Gtwo (\ep_1) \}$ and   
$\Om_+:= \{ \ep \in  [0,\de_3) \times [0,\de_4)  :   \Gtwo (\ep_1) < \ep_2 \}. $

Then the following statements hold:
\begin{itemize}
\item[(i)] If $\ep \in \Om_-$, then $H (\cdot;\ep)$ is an exponentially stable steady state of the restriction of the dynamical system \eqref{eq:PDE2} to the invariant hyperplane $\Hyp (C(\ep))$.
\item[(ii)] If $\ep \in \Om_+$, then the steady state $H (\cdot;\ep)$ of the restriction of \eqref{eq:PDE2} to the invariant hyperplane $\Hyp (C(\ep))$ is not stable in the sense of Lyapunov.
\item[(iii)]
Let $\Ecal \colon (s^-,s^+) \to \Rbb^2$ be a simple $C^3$-path that crosses 
the graph $\Gr \Gtwo$ nontangentially at certain point 
$\Ecal (s_0) $
 passing from the region $\Om_-$
to the region $\Om_+$.
Then, for the reduced system \eqref{e:PDEHyp2},
\begin{equation} \label{e:HBfTh}
\text{a Hopf bifurcation takes place at $(\zero;\Ecal(s_0))$ on the path $\Ecal$ }
\end{equation}
in the sense of  Definition \ref{d:Hbpath}. 
In particular, for arbitrary $\de>0$ there exists $s \in (s_0 - \de,s_0+\de)$ such that, for $\ep=\Ecal (s)$, 
\begin{multline}
\text{equation \eqref{eq:PDE2}  has a nonstationary periodic in $t$ solution $h(t,\th)$  } 
\\  \text{  that satisfies $ \sup\limits_{\substack{\theta \in S^1 \\ t \in \Rbb}} |h(t,\theta) - H (\theta;\Ecal (s))| < \de$.} \label{e:periodic}
\end{multline}
\end{itemize}
\end{theorem}

This theorem is proved in Section \ref{s:Stability}.


\section{Multi-parameter Hopf bifurcations\label{s:mpHopfB}}

This section considers multi-parameter Hopf bifurcations in the abstract infinite-\linebreak{}dimensional settings and provides the theoretical background for the proofs of the results of Section \ref{s:MainRes}. 

Bifurcation theory distinguishes between several types of criticality effects. These include, in particular:
\begin{itemize}
	\item[(B1)] spectral (or linear) criticality, where the spectrum of  an operator linearized around an equilibrium corresponds to the threshold between stability and instability;
	\item[(B2)] dynamic criticality, i.e., the situation when the actual change in the behavior of the dynamical system happens under a modification of a parameter through its specific critical value. 
\end{itemize}
In the finite-dimensional context of Arnold's lectures \cite{A72}, situations of type (B2) fit  the notion of a bifurcation. Situations of type (B1) usually play the role of necessary spectral conditions for specific bifurcations of type (B2).

We are interested in Hopf bifurcation 
(Andronov--Hopf bifurcation in the terminology of \cite{AAISh94,K95})
which has a variety of definitions depending on the dimension $\dim \Xcal$ of the state space $\Xcal$ and the dimension $n$ of the parameter space 
(see, e.g., \cite{A72,AAISh94,K95,DvGLW12} 
for the case $\dim \Xcal < \infty$ and  $n \in \Nbb$, 
\cite{CR77,K04,DvGLW12,L12} for the case $\dim \Xcal = \infty$ and $n=1$, 
and the discussion of the case $\dim \Xcal = \infty$ and $n \in \Nbb$ in \cite{K04,DK13}). 
The corresponding change of the long-time dynamics can be described 
as a branching of a family of nontrivial periodic solutions from a curve of equilibria 
at the point where a conjugate pair of eigenvalues of a linearized operator crosses the stability threshold $\ii \Rbb$.
This phenomenon was observed for $2$-dimensional systems already by Poincaré and studied by Andronov (see \cite{A72,AAISh94,K95}). 
The one-parameter Hopf bifurcation \cite{H42} describes an important  particular case of such branching where an additional nondynamical condition is satisfied (Hopf condition); namely, the conjugate pair of eigenvalues crosses $\ii \Rbb$ with the nonvanishing real part of their derivative. Roughly speaking, the Hopf bifurcation can be seen as a pitchfork bifurcation of a curve of cycles from a curve of equilibria.

In this section, we adapt the definition of multi-parameter Hopf bifurcation suggested for delayed  ODEs by Diekmann and Korvasova \cite{DK13} to the case of a dynamical system in an infinite-dimensional state space in such a way that it fits the specifics of the rimming--flow equation \eqref{eq:PDE}. This definition, roughly speaking, reduces the multi-parameter Hopf bifurcation to the one-parameter case considering paths $\Ecal(s)$, $s \in [s_-,s_+]$, in the parameter space $\Rbb^n$ that  cross 
$(n-1)$-dimensional bifurcation surfaces 
in $\Rbb^n$ at certain points $\ep^0 $ which are critical in the sense of (B1), i.e., in the sense of  the spectrum of the linearized operator $A$ (we call such points spectrally critical or $\si(A)$-critical).
In these settings, $s \in [s_-,s_+]$ plays the role of the 1-dimensional  parameter of the  one-parameter Hopf-bifurcation theorems in the style of \cite{CR77,K04,L12}. 

Consider an abstract evolution equation 
\begin{gather} \label{e:Feq}
\pa_t x = F (x;\ep) , \quad x \in U  \subset \Xcal, 
\end{gather}
dependent on an $n$-parameter $ \ep = (\ep_j)_{j=1}^n \in V \subset \Rbb^n $, 
where the set $V$ is open in $\Rbb^n $, and $U$ is an open subset of a real Banach space $\Xcal$.
It is assumed that a real Banach space $\Zcal$ and the map $F\colon U \times V \to \Zcal$ satisfy the following hypotheses:
\begin{gather}
      \text{$\Xcal$ is continuously and densely embedded in $\Zcal$},
	\label{h:HComp} 
	\\
	F : U \times V \to \Zcal
	\quad 
	\text{is a real analytic map}.
	\label{h:HAn}
\end{gather}
The derivative $\pa_t x$ is understood in the sense of the convergence w.r.t. the norm $\|\cdot \|_{\Zcal}$ of the Banach space $\Zcal$.

The \emph{complexification} $\Ycal_\crm$ of a real Banach space $\Ycal$ is defined as a formal sum $\Ycal_\crm := \Ycal + \ii \Ycal$ (in particular, $(\Rbb^n)_\crm = \Cbb^n$ is the complexification of the parameter space).
Then, for $y \in \Ycal_\crm$, the real part $\Re y$, the imaginary part $\Im y$, and the complex conjugation $\bar{y}$  of any vector $y \in \Ycal_\crm$ are well-defined.  If $\Ycal$ is an inner product space, the real inner product $( \cdot | \cdot )_\Ycal$ of $\Ycal$ is extended in a natural way to the complex inner product $( \cdot | \cdot )_{\Ycal_\crm}$.

Similarly, a real (linear) operator $T \in \Lcal (\Xcal,\Zcal) $ is extended in a natural way to a complex operator $T_\crm \in \Lcal (\Xcal_\crm,\Zcal_\crm) $.
Note that, if $\la \in \Cbb$ is an eigenvalue of $T_\crm$ corresponding to an eigenvector $x \in \Xcal_\crm$ (i.e., $ T_\crm x = \la x $ and $x \neq 0$), then the complex conjugate $\bar{\la}$ also belongs to the set of eigenvalues $\si_p (T_\crm)$ of $T_\crm$ and 
$T_\crm \bar{x} = \bar{\la} \bar{x}$.

Denoting by $D_x F(x^0;\ep^0)$
the Fréchet derivative of $F$ w.r.t. $x$ at $(x^0;\ep^0) \in U \times V$, one obtains a real analytic map 
$(x,\ep) \mapsto D_x F (x;\ep)$ acting from 
$U \times V$ to $\Lcal(\Xcal,\Zcal)$.
For each $(x;\ep) \in U \times V$, we define the operator 
\begin{equation} \label{e:A} 
	A(x;\ep) \in \Lcal(\Xcal_\crm,\Zcal_\crm) \ \text{as the complexification $(D_x F(x;\ep))_\crm$ }
\end{equation}
of $D_x F(x;\ep)$.
In the sequel, the complexified $x$-derivative $A(x;\ep)$ is considered also as an operator in the space $\Zcal_\crm$ with  domain $\dom A(x;\ep) = \Xcal_\crm$, which corresponds to the notation 
$A(x;\ep)\colon \Xcal_\crm \subseteq \Zcal_\crm \to \Zcal_\crm$.

\subsection{Hypersurfaces of spectrally critical $n$-parameters\label{s:Hypersurf}}

We define and study in this  subsection the family of spectrally critical $n$-parameters $\ep$ using the theory of multi-parameter perturbations of eigenvalues. The necessary results on multi-parameter perturbations are collected in Appendix \ref{a:A}. We prove in this  subsection Proposition \ref{p:SpCrManif1}, which  is crucial for the extension of the multi-parameter Hopf-bifurcation theory to the case of infinite-dimensional state space $\Xcal$.  Proposition \ref{p:SpCrManif1} shows that the combination of the spectral criticality with the sectoriality properties of $A(x;\ep)$ is robust under suitable multi-parameter perturbations of $(x;\ep)$. This later allows us to define rigorously Hopf-bifurcation manifolds via Proposition \ref{p:SpCrManif2} and Theorem \ref{t:n-parHb}.

A point $(x;\ep) $ is said to be an \emph{equilibrium point} if $x$ is an equilibrium of 
$\pa_t x = F (x;\ep)$, i.e., if  $F(x;\ep)=0$.
Accordingly,
$F^{-1}(\{0\}) = \{(x;\ep) \in U \times V \ : \ F(x;\ep)=0\}$ is the set of all equilibrium points of the evolution equation $\pa_t x = F (x;\ep)$.

Let $(x^0;\ep^0)$ be an equilibrium point such that $0 \in \rho (A(x^0;\ep^0))$,  i.e., such that
$A(x^0;\ep^0) $
is boundedly invertible as a densely defined operator in $\Zcal_\crm$. 
Then 
the analytic implicit function theorem (see, e.g., \cite[Appendix B]{PT86}) implies that there exists an open neighborhood $\Ncal^{x^0} \times \Ncal^{\ep^0} \subset U \times V $ of $(x^0;\ep^0)$ such that all equilibria $x \in \Ncal^{x^0}$ of $\pa_t x = F (x;\ep)$ with $\ep \in \Ncal^{\ep^0}$ 
have the form $x = X^\loc(\ep)$, where $X^\loc\colon \Ncal^{\ep^0} \to \Ncal^{x^0}$ is a certain real analytic map satisfying $x^0 = X^\loc(\ep^0)$. 
So $ X^\loc (\cdot)$ provides  a local analytic parametrization of the set of all equilibrium points near $(x^0;\ep^0)$. In other words, the graph 
$\Gr X^\loc $  of $X^\loc(\cdot)$, \linebreak
$\Gr X^\loc =
	\{(X^\loc(\ep);\ep) \ : \ \ep \in \Ncal^{\ep^0}\} $, 
is a \emph{local equilibrium manifold} near $(x^0;\ep^0)$.

Taking into account the stability of the property $0 \in \rho(A(x;\ep))$  (see Remarks \ref{r:PerRes}--\ref{r:PerSectC} in Appendix \ref{a:A}), one sees that the set of all equilibrium points where $A(x;\ep)$ is invertible 
\begin{gather} \label{e:Meq}
	\{(x;\ep) \in F^{-1} (\{0\})\ : \ 0 \in \rho(A(x;\ep))\}
\end{gather}
is an $n$-dimensional real analytic manifold
in $\Xcal \times \Rbb^n$.


\begin{definition}[cf. \cite{K04,L12}] \label{d:SpStability}
An equilibrium point $x$ of the evolution equation $\pa_t x = F (x;\ep)$ is called
\begin{itemize}
	\item[(i)] \emph{spectrally exponentially stable} if 	the spectrum $\si(A(x;\ep))$ of the complexified  Fréchet derivative of $F$ is a subset of the halfplane 
	$ \{\zeta \in \Cbb \ : \ \Re \zeta  \le  \xi\}$ for a certain $\xi<0$;
	\item[(ii)] \emph{spectrally unstable} if $\si(A(x;\ep)) \not \subseteq \{\zeta \in \Cbb \ : \  \Re \zeta \le 0 \}$.
\end{itemize}
\end{definition}

\begin{remark}
Under a variety of assumptions, a connection of these spectral properties with actual instability and exponential stability of 
$\pa_t x = F (x;\ep)$ is given by the Principle of Linearized Stability (see \cite{K04,L12}), which is considered in Section \ref{s:Stability} in the context of the reduced rimming-flow equation  \eqref{e:PDEHyp2}.
\end{remark}


\begin{remark} Keeping in mind the applications of Section \ref{s:MainRes}, we restrict ourselves  to the simplest spectrally critical case where exactly one pair of conjugate  eigenvalues lies on $\ii \Rbb$. If more eigenvalue pairs on $\ii \Rbb$ are allowed, one has to pay attention to possible additional  singularities caused by resonances (cf. the 1-parameter formulations in \cite{K04,L12}).
\end{remark}

\begin{definition}[cf. \cite{K95,K04,DvGLW12,L12}] \label{d:SpCr}
An equilibrium point $(x;\ep) \in U \times V$ is said to be  $\sigma(A)$-critical (or spectrally critical in the sense of Poincaré--Andronov) if the following two conditions are satisfied:
\begin{itemize}
	\item[(i)]  $\si(A(x;\ep))$ lies in the closed left half-plane $\{\zeta \in \Cbb \ : \  \Re \zeta \le 0 \}$;
	\item[(ii)] the pure imaginary part  $\si(A(x;\ep)) \cap \ii \Rbb$ of the spectrum consists of a pair of complex conjugate simple isolated eigenvalues, which are denoted by $\la_0^\pm$ in such a way that  $\Im \la_0^+ = - \Im \la_0^- >0$.
\end{itemize}
In this case, $\ep$ is called a $\sigma(A)$-critical $n$-parameter (associated with 
$x$).
\end{definition}

Let $(x^0;\ep^0)$ be a $\sigma(A)$-critical equilibrium point. Using Definition \ref{d:SpCr} (ii), we define eigenvalues $\la_0^\pm \in \ii \Rbb$ of $A_0 := A(x^0;\ep^0)$ and infer that $0 \in \rho(A_0)$.  Consequently,
\begin{equation*}
	(x^0;\ep^0) 
	\in 
	\left\{  (x,\ep) \in F^{-1} (\{0\}) \, : \, 0 \in \rho (A(x;\ep))  \right\}
\end{equation*}
and there exists  a local analytic parametrization $X^\loc : \Ncal_0^{\ep^0} \to \Ncal_0^{x^0}$ of all equilibrium points in a certain neighborhood $\Ncal_0^{x^0} \times \Ncal_0^{\ep^0}$ of $(x^0;\ep^0)$. The restriction of the $x$-derivative $D_x F(\cdot;\cdot)$ to 
the  local equilibrium manifold $\Gr X^\loc$ can be written as a real analytic function of  $\ep$, which  yields after the complexification 
the $\Lcal (\Xcal_\crm,\Zcal_\crm)$-valued  function \ 
	$A^\loc (\ep) := A( X^\loc (\ep); \ep)$ \ depending only on  $\ep \in \Ncal_0^{\ep^0}$.

It follows from Proposition \ref{p:PerSimpleC} and Remark \ref{r:PerSectC} that 
there exist an open
 neighborhood $\Ncal_{1,\crm}^{\ep^0} \subset  \Cbb^n$  of $\ep^0$ and  functions $\la^\pm(\cdot):\Ncal_{1,\crm}^{\ep^0} \to \Cbb$ such that 
\begin{align} \label{e:lapm0}
& \text{ $\la^\pm(\cdot)$ are holomorphic in $\Ncal_{1,\crm}^{\ep^0}$, \qquad $\la^\pm (\ep^0) = \la_0^\pm$,}
\\
&  \text{and $\la^\pm (\ep)$ are simple isolated eigenvalues of 
$A^\loc (\ep)$ for all $\ep \in \Ncal_1^{\ep^0} $,}  \label{e:lapm1}
\end{align}
where  $\Ncal_1^{\ep^0} := \Ncal_0^{\ep^0} \cap \Ncal_{1,\crm}^{\ep^0}$ is an $\Rbb^n$-neighborhood  of $\ep^0$.

Without loss of generality, the neighborhood $\Ncal_{1,\crm}^{\ep^0}$ can be chosen in such a way that   
 the Taylor series  \ $ \sum_{\alpha}\la_\alpha^+(\ep-\ep^0)^\alpha $ of $\la^+(\ep)$ 
converges in  $\Ncal_{1,\crm}^{\ep^0}$.
Here and below the summations are done w.r.t. the multi-parameter $\al \in \Nbb_0^n$, where 
$\Nbb_0 =  \{0\} \cup \Nbb$. 
In the $\Rbb^n$-neighborhood $\Ncal_1^{\ep^0}$ of $\ep^0$, we define the function 
\begin{gather} \label{e:r}
\rla (\ep) = \Re \la^\pm (\ep),
\end{gather}
 which is real analytic since \ $\rla (\ep) = \sum_\alpha (\Re \la_\al^+) (\ep - \ep^0)^\al  $ \ in $\Ncal_1^{\ep^0}$.


The condition that $A^\loc(\ep^0)$ generates a quasi-bounded holomorphic semigroup (in the terminology of Kato \cite{Kato}) is one of the conditions used in the infinite-dimensional stability theory \cite{CR77,K04,H06,L12} to ensure that, except the two eigenvalues $\la^\pm (\ep)$, the spectrum of $A^\loc (\ep)$  stays far away from the pure imaginary line $\ii \Rbb$ for $\ep$ close to $\ep^0$. 
We use this condition in the multi-parameter settings and adapt for this purpose the perturbation results of \cite{HillePhillips,Kato}.

Under $\Hsec$-sectorial operators we understand generators of quasi-bounded ho\-lo\-morphic semigroups $e^{t T}$ (essentially, they are sectorial operators in the sense of Henry \cite{H06}).
That is, for angles $\vphi \in (0,\pi]$, we consider sectors of the form
\[
     \Sigma_{\vphi} 
    = 
    \left\{
    \lambda \in \C\setminus\{0\} \ : \ \ |\arg (\lambda)| < \vphi
    \right\},
\]
where the complex argument $\arg (\lambda)$ of $\lambda \in \C\setminus\{0\}$ is defined as a unique $\vphi_0 \in [-\pi,\pi)$ such that $\lambda = r \ee^{i\vphi_0}$ for a certain $r \in \Rbb_+ := (0,+\infty)$. Let $\om \in (0,\pi/2]$. Slightly modifying the notation of \cite{Kato}, let us denote  by $\Hsec (\omega,0)$ the family  of densely defined closed operators $T$ in a complex Banach space such that the resolvent set $\rho (T)$ contains the sector $ \Sigma_{\frac{\pi}{2} + \omega}$ and, for any smaller sector 
$ \Sigma_{\pi/2 + \omega'}$ with $\omega' \in [0,\omega)$, the  uniform estimate
$
	\|(T-\zeta I)^{-1}\|
	\le 
	\frac{M_{\omega'}}{|\zeta|}, $
	\
	$
	\zeta \in \Sigma_{\frac{\pi}{2} + \omega'}, 
$
holds with a constant $M_{\omega'} >0$ depending on $\omega'$.
An operator $T$ belongs to the family $\Hsec (\omega,\zeta_0)$ if $T=T_0 + \zeta_0$ with a certain $T_0 \in \Hsec (\omega,0)$ and certain $\zeta_0 \in \Rbb$. Finally, we introduce the notation  
\begin{equation} \label{e:Hsec}
	\Hsec := \bigcup_{\substack{\zeta_0\in \Rbb \\ \om \in (0,\pi/2]}} \Hsec (\omega,\zeta_0) \ , 
\end{equation}
and we say that an operator $T$ is $\Hsec$-\emph{sectorial} if $T \in \Hsec$.

\begin{remark}\label{r:sectorial}
The family of $\Hsec$-sectorial operators is exactly \cite{HillePhillips,Kato,L12} the family of generators of quasi-bounded holomorphic semigroups.
 In the theory of evolution equations in Banach spaces, operators satisfying various modifications of the above spectral definition are simply called sectorial, cf. \cite{EN00,H06,L12}. Unfortunately, the theory of operators in Hilbert spaces reserves the  attribute `sectorial'  for a different class of operators (see the monograph of Kato \cite{Kato}). The class of sectorial operators in the sense of \cite{Kato} is defined via the numerical range and only partially overlaps with the family $\Hsec$.
\end{remark}

The next proposition essentially says that, if $A_0 = A(x^0;\ep^0)$ is an $\Hsec$-sectorial operator in $\Zcal_{\crm}$, then the set of $\sigma(A)$-critical $n$-parameters coincides locally with the \emph{zero-locus} $\rla^{-1} (\{0\}) = \{ \ep \in \Ncal_1^{\ep^0}\  : \ \rla(\ep) = 0\}$ of $\rla(\cdot)$.

\begin{proposition} \label{p:SpCrManif1}
Let $(x^0;\ep^0) $ be a $\si(A)$-critical point and let $A(x^0;\ep^0) \in \Hsec$. Then:
\begin{itemize}
	\item[(i)] $A^\loc(\ep)$ is $\Hsec$-sectorial for $\ep$ in a certain open $\Rbb^n$-neighborhood $\Ncal_2^{\ep^0}$ of $\ep^0$.
	\item[(ii)] There exists an open $\Rbb^n$-neighborhood $\Ncal_3^{\ep^0} $ of $\ep^0$ such that,  for all $\ep \in \Ncal_3^{\ep^0}$,
	the equilibrium point $(X^\loc(\ep);\ep)$ is $\si(A)$-critical  if and only if $\rla(\ep) = 0$.	
\end{itemize}
\end{proposition}


\begin{proof} 
Since $(x^0;\ep^0)$ is $\si(A)$-critical and $A_0 := A(x^0;\ep^0)$ is $\Hsec$-sectorial, there exists $\zeta_1<0$ such that the two simple eigenvalues $\la_0^\pm$ are the only points of $\si(A_0) $ in the half-plane $\{ \Re \zeta > \zeta_1 \}$. 
Using Remark \ref{r:PerSectC}  and Proposition \ref{p:PerSectC} one can choose the neighborhoods $\Ncal_2^{\ep^0}$ and $\Ncal_3^{\ep^0}   \subset \Ncal_2^{\ep^0} \cap \Ncal_1^{\ep^0} $ with the desired properties. This completes the proof. 
\end{proof}

\begin{definition} \label{d:r-regular}
Let $(x^0;\ep^0)$ be an equilibrium point. We say that $(x^0;\ep^0)$ is an \emph{$\rla$-regular equilibrium point} with an $\rla$-regular $n$-parameter $\ep^0$ if $\ep^0$ is a regular point of  $\rla$, i.e., if $\nabla \rla(\ep^0) \neq 0$, where$\nabla \rla(\ep^0) = (\pa_{\ep_1} \rla (\ep^0),  \pa_{\ep_2} \rla (\ep^0), \dots, \pa_{\ep_n} \rla(\ep^0))$ is the gradient at $\ep^0$ of the function $\rla$ introduced in \eqref{e:r}.
If $\nabla \rla(\ep^0) = 0$, the equilibrium point $(x^0;\ep^0)$ and the associated $n$-parameter $\ep^0$ are said to be \emph{$\rla$-singular}.   
\end{definition}

Recall that a \emph{hypersurface} in an  $n$-dimensional manifold is by definition an  $(n-1)$-dimensional submanifold. 

\begin{proposition} \label{p:SpCrManif2}
Let $\wh \Mcal_\Hsec^\rregH $ be the set of all $\rla$-regular $\si(A)$-critical equilibrium points $(x;\ep)$
such that $A(x;\ep) $ is $\Hsec$-sectorial. 
Then $\wh \Mcal_\Hsec^\rregH $ is an analytic $(n-1)$-dimensional manifold in $\Xcal \times \Rbb^n$.
\end{proposition}
\begin{proof}
Let $(x^0;\ep^0)$ be an $\rla$-regular equilibrium point.
Locally in the neighborhood $\Ncal_3^{\ep^0}$ from Proposition \ref{p:SpCrManif1}, operators $A^\loc (\ep)$ are $\Hsec$-sectorial and the set of 
$\si(A)$-critical $n$-parameters $\ep$ coincides with $\rla^{-1} (\{0\}) \cap \Ncal_3^{\ep^0}$. 
Applying the implicit function theorem at $\rla$-regular $n$-parameters $ \ep \in \rla^{-1} (\{0\}) \cap \Ncal_3^{\ep^0}$, we see that the set 
\[	\Mcal_{\Hsec,\loc}^\rregH 
	:= 
	\{ \ep \in \Ncal_3^{\ep^0}  :   (X^\loc(\ep);\ep) \text{ is } \si(A)\text{-critical}, \ \nabla \rla(\ep) \neq 0, \ A^\loc(\ep) \in \Hsec \} 
\]	
is an analytic hypersurface in the parameter space $\Rbb^n$.
Globally, this implies the statement of the proposition. 
\end{proof}

\subsection{Hopf bifurcations on paths\label{s:HBon Paths}}

We need a formalization of the notion of continuous families of \emph{cycles}  (confer Hopf-bifurcation theorems in \cite{K04,L12}).  
A \emph{cycle} $\psi(\cdot)$ of the dynamical system $\pa_t x = F(x;\ep)$ is its  periodic solution with the corresponding period $\Pf=\Pf(\psi)$ defined by $\Pf=\inf \{s\in (0,+\infty)  :  \psi (t) = \psi (t+s), \ t \in \Rbb\}$ assuming that periodic solutions produced by phase-shifts $t \mapsto t-t_0$ are considered to be identical. 
It is convenient to consider equilibria  as \emph{degenerate cycles} of period $0$. 

A \emph{simple path} is an injective continuous  function $\Ecal : \Ical \to \Ycal$ defined on a nondegenerate interval $\Ical \subset \RR$, where $\Ycal$ is a certain Banach space.
We use also a particular type of parametrized closed curves without self-intersections, which we call simple loops. A \emph{simple loop}  is understood in this paper as an injective continuous map $\ga : S^1 \to \Ycal$,
where $S^1$ as before is the unit circle $S^1 = \Rbb \ (\text{mod}\, 2\pi)$.

A simple loop $\wt \psi \in C(S^1;\Xcal) \cap  C^1(S^1;\Zcal) $ is called an \emph{$S^1$-parametrization of a cycle} $\psi$ of period $\Pf =\Pf(\psi)>0$ if 
$\wt \psi (2\pi  (t-t_0)/\Pf ) = \psi (t)$ for all $t \in  [t_0,t_0+\Pf] $
 with a certain $t_0 \in \Rbb$.
An $S^1$-parametrization $\wt \psi$ of a degenerate cycle $\psi \equiv x^0$ 
(i.e., an $S^1$-parametrization of an equilibrium point $x^0$) is defined as the constant function $\wt \psi (s)= x^0$,  $s\in[0,2\pi]$.  
 
Consider a function $\wt \ep\colon [\tau_-,\tau_+] \to V \subset \Rbb^n$. One  says that \[\{\psi_\tau , \wt \psi_\tau (\cdot), \wt \ep (\tau)\}_{\tau \in [\tau_-,\tau_+]}\] is a \emph{continuous family of cycles} of $\pa_t x = F(x;\ep)$ \emph{passing at $\tau_0 \in (\tau_-,\tau_+)$ through an equilibrium point $(x^0;\ep^0)$} if the three following conditions are fulfilled:
\begin{itemize}
	\item[(C0)]  $\wt \psi_\tau (\cdot)$ is an $S^1$-parametrization of a cycle $\psi_\tau $ of $\pa_t x = F(x;\wt \ep(\tau))$;
	\item[(C1)]  $\wt \ep(\tau_0) = \ep^0$  and 
	$\psi_{\tau_0} $ is the degenerate cycle with $\wt \psi_{\tau_0}(t) \equiv x^0$;
	\item[(C2)] $\wt \ep \in C ([\tau_-,\tau_+];V)$ and $\Psi \colon [\tau_-,\tau_+]\times [0,2\pi]\to \Xcal$ defined by $\Psi (\tau,s)= \wt \psi_\tau (s)$ is a homotopy. 
\end{itemize}

\begin{definition} \label{d:AHbpath}
Let $\Ecal\colon (s^-,s^+) \to V \subset \Rbb^n$ be a simple path such that, for a certain $s_0 \in (s^-,s^+)$, the value $\ep^0= \Ecal(s_0) $ is a $\si(A)$-critical $n$-parameter associated with an equilibrium $x^0 = X^\loc(\ep^0)$ of  $\pa_t x = F(x;\ep^0)$.
We say that a Poincaré--Andronov bifurcation takes place at $(x^0;\ep^0)$ on the path $\Ecal$  if there exists 
a subinterval $(s_-,s_+) \subset (s^-,s^+)$ with $s_0 \in (s_-,s_+)$ having the following properties:
\begin{itemize}
	\item[(i)] for each $s \in (s_-,s_0)$, the state $X^\loc (\Ecal (s))$ is a spectrally exponentially stable  equilibrium of $\pa_t x = F (x;\Ecal(s))$;
	\item[(ii)] for each $s \in (s_0,s_+)$, the state $X^\loc(\Ecal(s))$ is a spectrally unstable equilibrium of $\pa_t x = F (x;\Ecal(s))$;
	\item[(iii)] there exist a continuous function $\wt s\colon [\tau_-,\tau_+] \to (s_-,s_+)$ and 
a continuous family of cycles 
	 	 $\{\psi_\tau, \wt \psi_\tau (\cdot), \Ecal(\wt s(\tau))\}_{\tau  \in [\tau_-,\tau_+]}$ 
 	passing at $\tau_0 $ through $(x^0;\ep^0)$
 	so that the cycle $\psi_\tau$ is nondegenerate for $\tau \neq \tau_0$.
\end{itemize}
\end{definition}

A Hopf bifurcation w.r.t. a 1-dimensional parameter is a special type of Poincaré--Andronov bifurcation with an additional condition on the behavior of the perturbed eigenvalues $\la^\pm$, which can be adapted to bifurcations on paths in the following way.

\begin{definition}[Hopf transversality condition] \label{d:HopfTr}
Let $\Ecal\colon (s^-,s^+) \to V \subset \Rbb^n$ be a simple path such that $(x^0;\ep^0) = (x^0;\Ecal(s_0)) $ is a $\si(A)$-critical point for a certain $s_0 \in (s^-,s^+)$. Let $\la^\pm (\cdot)$ be the two analytic functions defined in a vicinity of $\ep^0$ by \eqref{e:lapm0}--\eqref{e:lapm1} such that $\la^+ (\ep)$ and $\la^- (\ep)=\overline{\la^+ (\ep)}$ are two simple isolated eigenvalues of $A^\loc (\ep)$ with the property that $\la^\pm_0 = \la^\pm (\ep^0) $ belong to $ \ii \Rbb \cap \si (A (x^0;\ep^0))$. We say that, on the path $\Ecal$, \emph{the Hopf transversality condition is satisfied at} $(x^0;\Ecal(s_0))$   if 
 \begin{equation} \label{e:HopfTr}
\pa_s \Re \la^{\pm} (\Ecal(s))|_{s=s_0} >0 .
 \end{equation}
\end{definition}

Under certain assumptions in the 1-parameter case, 
the Hopf transversality condition implies  the existence of a bifurcation of Poincaré--Andronov--Hopf type, as well as a variety of additional properties of the branching curve of cycles (see the Hopf bifurcation theorems in \cite{K04,DvGLW12,L12}).
In Definition \ref{d:Hbpath} below, we adapt these properties to the $n$-parameter case combining them with the approach of \cite{DK13}.

Slightly modifying the notation of \cite{K04}, we denote by $ C^{0,1} (S^1;\Ycal)$ the  Banach space of the functions that are  Lipschitz continuous on $S^1$ with values in a Banach space $\Ycal$.
The Banach space $ C^{1,1} (S^1;\Ycal)$ consists of $C^1$-functions  $g:S^1 \to \Ycal$ such that the derivative $g'$ belongs to $ C^{0,1} (S^1;\Ycal)$ (the norms in $ C^{k,1} (S^1;\Ycal)$ are introduced in the standard way, see  \cite{K04}).

A continuous family $\{\psi_\tau , \wt \psi_\tau (\cdot), \wt \ep(\tau)\}_{\tau \in [\tau_-,\tau_+]}$ of cycles passing at $\tau_0 $ through an equilibrium point $(x^0;\ep^0)$ is said to be a \emph{$C^1$-family of cycles passing through $(x^0;\ep^0)$ at $\tau_0$} 
if the following conditions are fulfilled additionally to (C0)--(C2): 
\begin{itemize}
	\item[(C3)] there exists $p \in C^1 ([\tau_-,\tau_+];\Rbb)$ such that 
	$p(\tau)$ is equal to the period  $\Pf(\psi_\tau)$ of $\psi_\tau$ for all $\tau \neq \tau_0$; 
	\item[(C4)] $\wt \psi_\tau \in C^{1,1} (S^1;\Zcal) \cap C^{0,1} (S^1;\Xcal)$ for all $\tau \in [\tau_-,\tau_+]$; 
	\item[(C5)] the mapping $\tau \mapsto \{\wt \psi_\tau, \wt \ep (\tau)\}$
	defines  a $C^1$-path in the Banach space \linebreak $\left(C^{1,1}(S^1;\Zcal) \cap C^{0,1} (S^1;\Xcal)\right) \times \Rbb^n$. 
\end{itemize}

\begin{definition}[Hopf bifurcations on paths] \label{d:Hbpath}
Assume that on 
a simple $C^1$-path \linebreak $\Ecal \colon (s^-,s^+) \to \Rbb^n$ a Poincaré--Andronov bifurcation takes place at $(x^0;\ep^0)=(x^0;\Ecal(s_0))$. Assume also that on $\Ecal$ the Hopf transversality condition \eqref{e:HopfTr} holds at $(x^0;\ep^0)$ and that $\la^\pm_0 $ are the eigenvalues of $A (x^0;\ep^0)$ defined as in Definitions \ref{d:SpCr} and  \ref{d:HopfTr}.
Then we say that 
a \emph{Hopf bifurcation takes place at $(x^0;\ep^0)$ on  $\Ecal$} if for any interval $(s_-,s_+) \subseteq (s^-,s^+) $ containing $s_0$ there exist
 a $C^1$-function $\wt s \colon [\tau_-,\tau_+] \to (s_-,s_+)$ and
a $C^1$-family of cycles $\{\psi_\tau, \wt \psi_\tau(\cdot),\Ecal (\wt s (\tau))\}_{\tau  \in [\tau_-,\tau_+]}$
of $\pa_t x = F (x;\ep)$ 
passing at $\tau_0 = \frac{\tau_- +\tau_+}{2}$ through $(x^0;\ep^0)$ such that  for small enough $\de>0$ the following statements hold true:
\begin{enumerate}
	\item[(iv)] the `period function' $p (\tau)$ defined by (C3) satisfies $p (\tau_0) = \frac{2\pi}{\Im \la_0^+}$;
	\item[(v)] $\psi_{\tau_0 - \wt \tau} = \psi_{\tau_0 + \wt \tau}$, $\Ecal(\wt s(\tau_0 - \wt \tau)) = \Ecal(\wt s(\tau_0 + \wt \tau))$, and 
	 $\wt \psi_{\tau_0 - \wt \tau} (t )= \wt \psi_{\tau_0 +\wt \tau}  \left( t+ \pi  \right)  $  for all $\wt \tau \in (-\de, \de)$;
		\item[(vi)] for  $s \in (s_0-\de,s_0+\de)$, the existence of a nondegenerate cycle  $\psi$ of the dynamical system $\pa_t x = F (x;\Ecal(s))$ with 
an $S^1$-parametrization 
\[
\wt \psi \in \Bbb_{\de } \left(\one x^0;C^{1,1} (S^1;\Zcal) \cap C^{0,1} (S^1;\Xcal) \right) 
\]
and with the period  $\Pf (\psi) \in (p(\tau_0)-\de,p(\tau_0)+\de)$
implies 
$\psi =\psi_\tau$ \ and \ $s = \wt s (\tau)$ \ for a certain $\tau  \in [\tau_-,\tau_+]$  (i.e., implies the uniqueness of the branching curve of cycles).
\end{enumerate}
Here $\one x^0$ is the constant function equal to $x^0$ for all $t$. 
\end{definition}


\begin{definition} \label{d:HbPoint}
If a Hopf bifurcation takes place at $(x^0;\ep^0)$ on  a certain path, we  say that $(x^0;\ep^0)$ is a \emph{Hopf--bifurcation point} and that $\ep^0$ is the associated  \emph{Hopf-bifurcation $n$-parameter}.
\end{definition}

A more general notion of a \emph{Poincaré--Andronov--bifurcation point} can be defined similarly. Obviously, every 
Hopf--bifurcation point is also a Poincaré--Andronov--bifurcation point. The converse generally is not true.

\begin{remark} \label{r:RegNontan}
Assume that a Hopf bifurcation takes place on a path $\Ecal$ at $(x^0;\ep^0)=(x^0;\Ecal(s_0))$. Then $(x^0;\ep^0)$ is $\si(A)$-critical and $\rla(\ep^0)=0$. Moreover, the Hopf transversality condition \eqref{e:HopfTr} yields 
\begin{gather} \label{e:E'gradr} 
	\pa_s \rla(\Ecal(s))|_{s=s_0}
	= 
	\<\pa_s \Ecal(s_0),\nabla \rla(\ep^0)\>_{\Rbb^n}>0 
	\end{gather}
	and, in turn, $	\pa_s \Ecal (s_0) \neq 0 $ and $\nabla \rla(\ep^0) \neq 0$.
Hence, there exists a neighborhood $\Ncal^{\ep^0}$ of $\ep^0$ such that
$
	\Mcal^\loc 
	:=
	\{\ep \in \Ncal^{\ep^0} \ : \ \rla(\ep)=0, \ \nabla\rla(\ep) \neq 0\} $ 
	 is an analytic hypersurface in $\Rbb^n$ containing $\ep^0$.
Note that \eqref{e:E'gradr} implies that $\Ecal$ passes through the hypersurface $\Mcal^\loc$ at $\ep^0 = \Ecal(s_0)$ nontangentially.
\end{remark}

\subsection{Manifolds of Hopf-bifurcation points}

The aim of this subsection is to formulate rigorously the notion of Hopf-bifurcation manifold (Hopf-bifurcation curve in the case $n=2$) and to prove the corresponding multi-parameter Hopf-bifurcation theorem, which essentially says the set of the $\rla$-regular $\si(A)$-critical  points satisfying an additional $\Hbb$-sectoriality assumption is a Hopf-bifurcation manifold. A particular feature of this multi-parameter bifurcation theorem is that the Hopf transversality condition is essentially replaced by the $\rla$-regularity (see Remark \ref{r:RegNontan}).

A $C^1$-path $\Ecal\colon (s^-,s^+) \to \Rbb^n$ is said to be regular if $\pa_s \Ecal(s) \neq 0$ for all $s \in (s^-,s^+)$.

\begin{definition} \label{d:HbManifold}
An $(n-1)$-dimensional $C^1$-manifold $\wh \Mcal$ in $\Xcal \times \Rbb^n$ is called a \emph{Hopf-bifurcation manifold} if for every point $(x^0;\ep^0) \in \wh \Mcal$ there exist its $\Xcal \times \Rbb^n$-neighborhood $\Ncal^{x^0} \times \Ncal^{\ep^0} $ and a $C^1$-hypersurface $\Mcal^\loc$ in $\Ncal^{\ep^0} \subset  \Rbb^n$ with the following properties:
\begin{itemize}
	\item[(i)] $\wh \Mcal^\loc := \wh \Mcal \cap (\Ncal^{x^0} \times \Ncal^{\ep^0})$ is the set of all 
	Hopf-bifurcation points in $\Ncal^{x^0} \times \Ncal^{\ep^0}$.
	\item[(ii)] There exists a $C^1$-map $X^\loc\colon \Ncal^{\ep^0} \to \Ncal^{x^0}$ such that 
	$\wh \Mcal^\loc = \{(X^\loc(\ep);\ep) \ : \ \ep \in \Mcal^\loc\}$.
	\item[(iii)] On every simple regular $C^3$-path  $\Ecal\colon [s^-,s^+] \to \Rbb^n$ passing through $\ep^0$ nontangentially to the hypersurface $\Mcal^\loc$, a Hopf bifurcation takes place at $(x^0;\ep^0)$  after a possible change of orientation of $\Ecal$. 
\end{itemize}
If these conditions are fulfilled, we say that $\Mcal^\loc$ is a \emph{local (parametric Hopf-) bifurcation hypersurface} that parametrizes the \emph{local Hopf-bifurcation manifold} 
$\wh \Mcal^\loc$. In the case $n=2$, we also call a connected $1$-dimensional Hopf-bifurcation manifold $\wh \Mcal$ \  \emph{a Hopf-bifurcation curve} \ and call a connected local parametric Hopf-bifurcation hypersurface  $\Mcal^\loc$ \ \emph{a local parametric Hopf-bifurcation curve}.
\end{definition}

 
\begin{theorem}[multi-parameter Hopf-bifurcation theorem] \label{t:n-parHb}
Assume \eqref{h:HComp} and \eqref{h:HAn}. Let $\wh \Mcal_\Hsec^\rregH $ be the set of all $\rla$-regular $\si(A)$-critical equilibrium points $(x;\ep)$
such that $A(x;\ep) $ is $\Hsec$-sectorial. Then:
\begin{itemize}
	\item[(i)]  $\wh \Mcal_\Hsec^\rregH$ is a Hopf-bifurcation manifold.
	\item[(ii)] Assume that a $\si(A)$-critical point $(x;\ep)$ is $\rla$-singular (i.e., $\nabla \rla (\ep) = 0$). Then $(x;\ep)$ is not a Hopf-bifurcation point.
	\item[(iii)] $\wh \Mcal_\Hsec^\rregH$ is the set of all Hopf-bifurcation points $(x;\ep)$ such that 
	$A(x;\ep) $ is $\Hsec$-sectorial. 
\end{itemize}
\end{theorem}



\subsubsection{Proof of Theorem \ref{t:n-parHb}.\label{ss:ProofHb}}

\emph{Step 1. Proof of statement (ii).} If  $(x^0;\ep^0)$ is an $\rla$-singular $\si(A)$-critical point, then Remark \ref{r:RegNontan} implies that $(x^0;\ep^0)$ is not a Hopf-bifurcation point. 

\vspace{1ex}

\emph{Step 2.}  Let $(x^0;\ep^0)$ be an $\rla$-regular $\si(A)$-critical point such that $A_0 = A(x^0;\ep^0)$
is $\Hsec$-sectorial. Then it follows from Section  \ref{s:Hypersurf} that  there exist an $\Xcal \times \Rbb^n$-neighborhood $\Ncal^{x^0} \times \Ncal^{\ep^0}$ of  $(x^0;\ep^0)$ and an analytic map $X^\loc \colon \Ncal^{\ep^0} \to \Ncal^{x^0}$ such that $x^0 = X^\loc (\ep^0)$ and the following properties hold:
\begin{itemize}
	\item[(a)]  the manifold  $\Gr X^\loc 	=
	\{(X^\loc(\ep);\ep) \ : \ \ep \in \Ncal^{\ep^0}\}$ is the set of all equilibrium points in $\Ncal^{x^0} \times \Ncal^{\ep^0}$;
	\item[(b)] the Taylor series at $\ep^0$ of the real analytic function $\rla (\ep) = \Re \la^+ (\ep)$  converges in $\Ncal^{\ep^0}$;
	\item[(c)] the set \
	$
	\Mcal^\loc := \{\ep \in \Ncal^{\ep^0} \ :  \ (X^\loc(\ep);\ep) \text{ is } \si(A)\text{-critical and } \nabla \rla(\ep) \neq 0\}$ \ 
	is a connected analytic hypersurface in $\Ncal^{\ep^0}$ and, for every $\ep \in \Mcal^\loc$, the operator 
	$A^\loc(\ep)$ is $\Hsec$-sectorial;
	\item[(d)] the $(n-1)$-dimensional analytic manifold $\wh \Mcal^\loc := \wh \Mcal_\Hsec^\rregH \cap (\Ncal^{x^0} \times \Ncal^{\ep^0})$  coincides with 
	$\{(X^\loc(\ep);\ep)\  : \ \ep \in \Mcal^\loc\}$ and contains $(x^0;\ep^0)$.
\end{itemize}

Remark \ref{r:RegNontan} and statement (ii), which is already proved in Step 1, imply that  $\Mcal^\loc$ satisfies condition (ii) of Definition \ref{d:HbManifold} and that $\wh \Mcal^\loc$ contains all Hopf bifurcation points in $\Ncal^{x^0} \times \Ncal^{\ep^0}$.

\emph{Step 3. We now verify property (iii) of Definition \ref{d:HbManifold}.} This property implies that 
$\wh \Mcal^\loc$ is exactly the set of all Hopf-bifurcation points in $\Ncal^{x^0} \times \Ncal^{\ep^0}$, and so, implies  statements (i) and (iii) of Theorem \ref{t:n-parHb}.

Property (iii) of Definition \ref{d:HbManifold} follows from  Proposition \ref{p:HbOnPath} given below, which can be proved essentially by the reduction to the one-parameter arguments of Kielhöfer \cite[Section I.8]{K04}. 

\begin{remark} \label{r:Comp}
Note that we drop the compactness assumption of \cite[formula (I.8.8)]{K04} which occurs to be superfluous in our settings (compare \cite[Remark I.9.2]{K04} and \cite[Theorem 9.3.3]{L12}). Indeed,
the assumption of compactness of $e^{tA_0}$ for $t>0$ is used in the proof of \cite[Proposition I.8.1]{K04} only to prove \cite[statement (I.8.20)]{K04}, which is written below as  statement \eqref{e:sieA}. In our settings, this statement is valid without the compactness assumption and can be proved, e.g., using the combination of the spectral mapping theorem for holomorphic semigroups \cite[Corollary IV.3.12]{EN00} and 
 the spectrum decomposition theorem \cite[Theorem III.6.17]{Kato}. Essentially, the compactness assumption is replaced in our settings by the assumption that $\la_0^\pm$ are isolated simple eigenvalues of  $A_0 \in  \Hbb$.
\end{remark}

\begin{proposition} \label{p:HbOnPath}
Let $(x^0;\ep^0)$ be a
$\si(A)$-critical point such that $A_0 = A(x^0;\ep^0)$ is $\Hsec$-sectorial. 
Let $s_0 \in (s^-,s^+)$, and let $\Ecal\colon (s^-,s^+) \to \Rbb^n$ be a simple $C^3$-path with $\Ecal(s_0)=\ep^0$ such that  the following transversality condition holds true 
\begin{gather} \label{e:trGeom}
	\< \pa_s \Ecal(s_0) | \nabla \rla (\ep^0)  \>_{\Rbb^n} >0 .
\end{gather}
Then a Hopf bifurcation takes place at $(x^0;\ep^0)$ on the path $\Ecal$.
\end{proposition}

\begin{proof}[Proof of Proposition \ref{p:HbOnPath}.]
Without loss of generality, we can assume that $\Ecal(s) \in \Ncal^{\ep^0}$ for all $s \in (s^-,s^+)$ and consider the reduced dynamical system 
\begin{gather} \label{e:wtF}
	\text{$\pa_t x = \wt F(x;s)$ 
	\quad with \quad 
	$\wt F(x;s) := F \left(x+X^\loc(\Ecal(s));\Ecal(s)\right)$},
\end{gather} 
depending only on the $1$-dimensional parameter $s \in (s^-,s^+)$. The complexification of the Fréchet derivative $D_x \wt F(x;s)$ is equal to \[
\wt A(x;s) = A(x+X^\loc(\Ecal(s));\Ecal(s)).\] In particular, $(0;s^0)$ is an equilibrium point of \eqref{e:wtF} and $\wt A(0;s_0) =  A_0$.

Since $(x^0;\ep^0)$ is $\si(A)$-critical,  $\si(A_0) \cap \ii \Rbb$  consists of two isolated simple eigenvalues $\la^\pm_0 = \pm \ii \ka_0$ of  $A_0$, where $\ka_0>0$ is a certain number.
There exists a bounded domain $\Dcal \subset \Cbb$ with a positively oriented smooth Jordan curve as its boundary $\pa \Dcal$ satisfying $\si(A_0) \cap \Dcal = \{\la_0^+,\la_0^-\}$ and $\pa \Dcal \subset \rho(A_0)$.   
We denote by $P_0 \in \Lcal(\Zcal)$ the Riesz  projection 
$ \displaystyle 
	P_0 
	= 
	- \frac{1}{2\pi\ii} \oint_{\pa \Dcal} (A_0 - \la I_\Zcal)^{-1} \dd \la
$
corresponding to $\pa \Dcal$. 
The closed subspaces $ P_0 \Zcal$ and $(I_{\Zcal}-P_0) \Zcal$ are the images of the  projection $P_0$ and of the complementary  projection $I_{\Zcal}-P_0$, respectively. By the spectrum decomposition theorem \cite[Theorem III.6.17]{Kato} each of the subspaces $ P_0 \Zcal$ and $(I_{\Zcal}-P_0) \Zcal$ is an invariant subspace of  $A_0$, and the direct sum decomposition 
(not necessarily orthogonal)
$ \Zcal = P_0   \Zcal \ \oplus  \ (I_{\Zcal}-P_0) \Zcal $ takes place. These two invariant subspaces are also invariant for the quasi-bounded holomorphic semigroup $\{\ee^{tA_0}\}_{t \ge 0}$ (see Remark \ref{r:sectorial}). In particular, the restriction $(I_{\Zcal}-\ee^{2\pi \ka_0^{-1} A_0}) |_{(I_{\Zcal}-P_0) \Zcal}$ of $I_{\Zcal}-\ee^{2\pi \ka_0^{-1} A_0}$ to the invariant subspace  $(I_{\Zcal}-P_0) \Zcal$ is an operator in the subspace  $(I_{\Zcal}-P_0) \Zcal$. 

The following statement holds true: 
\begin{equation} \label{e:sieA}
\text{$(I_{\Zcal}-\ee^{2\pi \ka_0^{-1} A_0}) |_{(I_{\Zcal}-P_0) \Zcal} $ is a self-homeomorphism of $ (I_{\Zcal}-P_0) \Zcal $.}
\end{equation}
(compare \cite[Proposition 4.4.8]{L12}). Indeed,
 since $A_0 \in \Hsec$ and since $(x^0;\ep^0)$ is $\si(A)$-critical,  condition (ii) of Definition \ref{d:SpCr}
implies that there exists a constant $\zeta_0<0$ such that 
\begin{gather} \label{e:siAinf}
\text{$\si \bigl( A_0 |_{(I_{\Zcal}-P_0) \Zcal} \bigr) = \si(A_0) \setminus \{\la_0^+,\la_0^-\} $ is a subset of $ \{ \zeta \in \Cbb : \Re \zeta < \zeta_0\}$}.
\end{gather} 
Besides, 
 $A_0 |_{(I_{\Zcal}-P_0) \Zcal} $ is a generator of a quasi-bounded holomorphic semigroup in $(I_{\Zcal}-P_0) \Zcal$. Combining the spectral mapping theorem for  quasi-bounded holomorphic semigroups \cite[Corollary IV.3.12]{EN00} with \eqref{e:siAinf}, we see that $1 \not \in \sigma (\ee^{2\pi \ka_0^{-1} A_0} |_{(I_{\Zcal}-P_0) \Zcal)})$.
This completes the proof of \eqref{e:sieA}.
Now, the arguments of  the proof of \cite[Theorem I.8.2]{K04} on Hopf bifurcations w.r.t. a $1$-dimensional parameter can be applied to $\pa_t x = \wt F(x;s)$ 
(see Remark \ref{r:Comp}).
 \end{proof}

This completes the proofs of Proposition \ref{p:HbOnPath} and of Theorem \ref{t:n-parHb}.

\subsection{\label{ss:singular} The two-parameter case and bifurcation curves near $\rla$-singular points}

We have shown that the set of $\si(A)$-critical points is described locally by the equation $\rla(\ep) =0$. In the two-parameter case $n=2$, the branching of the curves of the zero-locus for $\rla(\cdot)$ can be studied with the use of the Newton polygon and Puiseux series, see  \cite{VT74,K04}.

Motivated by the  rimming-flow equation \eqref{eq:PDE2}, we consider a relatively simple case of an $\rla$-singular point with a nonzero Hessian determinant $\det \Hf_{\rla}$ of $\rla$. We  describe for this case the parametric Hopf-bifurcation curves contained in the  local zero-locus of $\rla$.

To this end, let $\ep^0 = (\ep^0_1,\ep^0_2) \in V \subset \Rbb^2$ be an $\rla$-singular $\si(A)$-critical 2-parameter associated with an equilibrium $x^0$. We denote by 
\begin{equation} \label{e:Hrla}
	\Hf_{\rla} (\ep^0) 
	= 
	\begin{pmatrix} 
	\pa_1^2 \rla(\ep^0)  & \pa_1 \pa_2 \rla(\ep^0) \\ 
	\pa_2 \pa_1 \rla(\ep^0) & \pa_2^2 \rla(\ep^0) \end{pmatrix} 
	= 
	\begin{pmatrix} 
	2 \rla_{2,0} & \rla_{1,1}  
	\\ 
	\rla_{1,1}  &  2 \rla_{0,2}  
	\end{pmatrix}
\end{equation} 
the Hessian matrix of $\rla$ at $\ep^0$ and assume that its determinant $\det \Hf_{\rla} (\ep^0) $ is nonzero. In this section, 
$\rla_{i,j}=\Re \la^+_{i,j}$ are the Taylor coefficients of $\rla (\cdot)$ at $\ep^0$ (see (\ref{e:r})) and  $\pa_j = \pa_{\ep_j}$.

Note that the case $\det \Hf_{\rla} (\ep^0) > 0 $ is trivial from the point of view of Hopf bifurcations since $(x^0;\ep^0)$ is an isolated $\si(A)$-critical point and $\rla (\cdot)$ preserves its sign in a sufficiently small punctured neighborhood of $\ep^0$. Hence, $(x^0;\ep^0)$ is not a  Poincaré--Andronov bifurcation point, and so, is not a Hopf--bifurcation point.

The following lemma, describing the local zero-locus of $\rla(\cdot)$ in the case $\det \Hf_{\rla}(\ep^0) <0$, can be obtained by a combination of several formulas of \cite[Section I.2.7]{VT74} which consider one-by-one the corresponding cases of the Newton diagram. By $\sign(\cdot)$ we denote the signum function. 

\begin{lemma}[\cite{VT74}] \label{l:n=2Det<0}
Let $r\colon V \to \Rbb$ be a real analytic function with 
\[\text{
$r(\ep^0) = 0, \nabla r(\ep^0) = 0$, and 
$\det \Hf_{r}(\ep^0) <0$, 
}\]	
where 
\ $	\Hf_r (\ep^0) 
	= 
	\begin{psmallmatrix}
	2 r_{2,0}   & r_{1,1}  
	\\ 
	r_{1,1}  &  2 r_{0,2}  
	\end{psmallmatrix}  
$  is the Hessian matrix of $r$. 
Then:
\begin{itemize}
	\item[(i)] There exists a neighborhood $\Ncal^{\ep^0}$ of $\ep^0$ so that the set
	$
	\{\ep \in \Ncal^{\ep^0}  :  \ r(\ep)=0\}
	$
	consists of the union $\Gamma^+ \cup \Gamma^-$ of two simple regular real analytic curves $\Gamma^\pm$ having a nontangential intersection at $\ep^0$ as their only  intersection point. 
	\item[(ii)] Regular parametrization \quad $\gamma^\pm\colon (s^\pm_1,s^\pm_2) \to \Rbb^2$ \quad of \ $\Gamma^\pm$ can be chosen such that $0 \in (s^\pm_1,s^\pm_2)$, \ $\gamma^\pm(0) = \ep^0$, \ and the tangential vectors at $\ep^0$ are 
\end{itemize}
	\begin{gather} \label{e:paGa-}
		\pa_s \gamma^- (0)  
		= 
		\begin{pmatrix} 
		-2 r_{0,2} 
		\\   
		r_{1,1} + \sign (r_{1,1}) \sqrt{-\det \Hf_r} \end{pmatrix},
		\\
				\pa_s \gamma^+ (0)  
		= 
		\begin{pmatrix} 
		r_{1,1} + \sign (r_{1,1}) \sqrt{-\det \Hf_r} \\ 
		-2 r_{2,0} 
		\end{pmatrix}. 
		\label{e:paGa+}
	\end{gather}
\end{lemma}

For the reader's convenience, an elementary proof of Lemma \ref{l:n=2Det<0} without the use of Newton diagrams is given in Appendix \ref{a:B}.

\begin{theorem} \label{t:n=2Det<0}
In the settings of Subsection \ref{s:Hypersurf}, assume that $n=2$ and  that the operator $A(x^0;\ep^0)$ is  $\Hsec$-sectorial at an $\rla$-singular $\si(A)$-critical point  $(x^0;\ep^0)$ with $\det \Hf_{\rla} (\ep^0) <0 $. 
Then there exists a neighborhood $\Ncal^{x^0} \times \Ncal^{\ep^0}$ of $(x^0;\ep^0)$ such that:
\begin{itemize}
	\item[(i)]  $A^\loc(\ep) \in \Hsec$ \ for all $\ep \in \Ncal^{\ep^0}$ \ and  \ 
 $	\{\ep \in \Ncal^{\ep^0}\ : \  r(\ep)=0 \}
	= 
	\Gamma^+ \cup \Gamma^- , \
$	
	where the curves $\Gamma^\pm$ and their parametrization $\ga^\pm$ are as in Lemma \ref{l:n=2Det<0}. In particular, \eqref{e:paGa-}--\eqref{e:paGa+} are satisfied.
	
	
	\item[(ii)] The set $\{(X^\loc(\ep);\ep)  :  \ep \in (\Gamma^+ \cup \Gamma^-) \setminus \{\ep^0\} \}$ is a $1$-dimensional Hopf-bifurcation manifold that contains all Hopf-bifurcation points in $\Ncal^{x^0} \times \Ncal^{\ep^0}$ and consists of four simple disjoint Hopf-bifurcation curves.
	\item[(iii)] The corresponding set $(\Gamma^+ \cup \Gamma^-) \setminus \{\ep^0\} $ of Hopf-bifurcation 2-parameters is a union of four simple disjoint regular real analytic curves.
\end{itemize}
\end{theorem}

\begin{proof} It follows from Proposition \ref{p:SpCrManif1} and Lemma \ref{l:n=2Det<0} that 
$\Ncal^{\ep^0}$ can be chosen such that statement (i) holds. 
By Theorem \ref{t:n-parHb}, all local Hopf-bifurcation 2-parameters $\ep=(\ep_1,\ep_2)$ near $\ep^0$ belong to $\Gamma^+ \cup \Gamma^-$ and are characterized in this set by the property that the gradient $\vb(\ep) := \nabla \rla(\ep)$ is nonzero. In particular, $\ep^0$ is not a Hopf-bifurcation 2-parameter.

Since 
$
\vb (\ep) = \begin{pmatrix} 
2 \rla_{2,0}   & \rla_{1,1}  \\ 
\rla_{1,1}  &  2 \rla_{0,2}  
\end{pmatrix}\begin{pmatrix} \ep_1 - \ep_1^0 \\ \ep_2 - \ep_2^0  \end{pmatrix}  + o (|\ep-\ep^0|)$ as $\ep \to \ep^0$,
we see from \eqref{e:paGa-}--\eqref{e:paGa+} that on the curves $\Ga^\pm$ the following asymptotic formulae hold:
\begin{gather*}
\vb (\gamma^- (s)) =  s \begin{pmatrix} -\det \Hf_\rla  + |\rla_{1,1}| \sqrt{-\det \Hf_\rla} \\ 2 \rla_{0,2}  \sign (\rla_{1,1}) \sqrt{-\det \Hf_\rla} \end{pmatrix}  + o(s) 
\ \text{ as $s \to 0$,}
\\
\vb (\gamma^+ (s))  = 
s \begin{pmatrix} 2 \rla_{2,0} \sign (\rla_{1,1}) \sqrt{-\det \Hf_\rla} \\ 
-\det \Hf_\rla + |\rla_{1,1}| \sqrt{-\det \Hf_\rla} 
\end{pmatrix} + o(s) \ \text{ as $s \to 0$}.
\end{gather*}
It follows from $\det \Hf_\rla  < 0$ that $\vb (\gamma^\pm (s)) \neq 0$ for $s \in (-\de,\de) \setminus \{0\}$ with a certain $\de>0$. This means that, after a possible shrinking of $\Ncal^{\ep^0}$, 
\[
	(\Gamma^+ \cup \Gamma^-) \setminus \{\ep^0\} 
	= 
	\{\ep \in \Ncal^{\ep^0} \ : \ \rla (\ep)=0, \ \nabla \rla(\ep) \neq 0, \  A^\loc(\ep) \in \Hsec\}
\]
and the set described by this equality is the set of all local Hopf-bifurcation parameters. 
This proves statement (iii) and, in turn implies statement (ii). 
\end{proof}

\begin{remark} \label{e:BifCurves}
For the  the  rimming-flow equation \eqref{eq:PDE2}, the physically relevant parameters $\ep_1$ and $\ep_2$ are positive. So for this equation from the four parametric bifurcation curves described in statement (iii) of Theorem \ref{t:n=2Det<0}, only the curve lying in the 1st quadrant is physically relevant, see Theorem \ref{t:HB}.
\end{remark}





\section{Proof of Theorem \ref{thm:solution_ODE}}
\label{s:steady}

We prove in this section the statements of Theorem \ref{thm:solution_ODE} in the wider  settings of complex Sobolev spaces $\Hcal_\crm^s (S^1) $ with an arbitrary regularity $s \ge 3$. For $s=4$, this complexified parametrization of steady states with $s=4$ will be needed in Section \ref{s:rfHcurves} for the calculation of perturbations of eigenvalues. 

Let $s \geq 3$. Recall that $b>0$ is a constant.
We introduce in $\Hcal^{s-3}_\crm (S^1)$ the linear operator $\Bo$
with the domain $\dom \Bo = \Hcal_\crm^s$ by the formula 
\begin{equation*} \label{eq:def_B}
	\Bo f = \Bo[f] := b(\partial_\theta f + \partial_\theta^3 f) .
\end{equation*}
Then the equation \eqref{eq:ODE} for the steady states $H =H(\cdot;\ep) = H (\cdot; \ep_1,\ep_2)$ can be written in the form 
\begin{equation} \label{eq:g}
	H - \tfrac{\ep_1 \cos \theta }{3}  H^3   + H^3 \Bo [H] + \ep_2 \sin(\theta) H^3 \pa_\te H   = \one .
\end{equation}
Applying the analytic implicit function theorem (see, e.g., \cite[Appendix B]{PT86}) in a small enough neighborhood of $H (\cdot; 0,0) = H_0 \equiv \one$,  we study $\Cbb$-valued  solutions $H(\cdot;\ep) \in \Hcal^s_\crm (S^1)$ to the steady state  equation (\ref{eq:g}) with small enough 2-parameters $\ep = (\ep_1,\ep_2) \in \Cbb^2$ (i.e., with small enough $\ep_j \in \Cbb$, $j=1,2$).  
As a by-product, we prove Theorem \ref{thm:solution_ODE}, which is concerned with the case $s=4$ for small real values of $\ep_1$ and $\ep_2$.

Since $b>0$ is a constant, the operator $\Bo$ is diagonal w.r.t. the orthogonal basis $\{e^{\ii n \te}\}_{n\in \Zbb}$ in $\Hcal^{s-3}_\crm (S^1)$ and has a purely imaginary spectrum $\si (\Bo) \subset \ii \Rbb$. Consequently, the operator $B+I$, where $I=I_{\Hcal^{s-3}_\crm (S^1)}$ is the identity operator in $\Hcal^{s-3}_\crm (S^1)$, is boundedly invertible in $\Hcal^{s-3}_\crm(S^1)$ and can be perceived as a linear homeomorphism from $ \Hcal^s_\crm(S^1)$ to $\Hcal^{s-3}_\crm(S^1)$; we denote this homeomorphism 
\begin{equation} \label{eq:def_G}
	 \quad G\colon \Hcal^s_\crm(S^1) \rightarrow \Hcal^{s-3}_\crm(S^1),
	\quad
	 G[f]  := (B+I)f.
\end{equation}

Recall that, for $r>0$, $\Dbb_r (0)=\Bbb_{r} (0;\Cbb) $ is an open complex disc with the center at $0$, and 
$\Dbb_r^2 (0) = \Dbb_r (0) \times \Dbb_r (0)$ is a polydisc.

\begin{proposition}\label{p:CsolutionODE}
For every natural number $ s \geq 3$, there exist positive  constants $\de_0$ and $\de_1$ (possibly, depending on $s$) with the following properties:
\begin{itemize}
	\item[(i)] For every $\ep \in \Dbb_{\de_0}^2 (0)$  there exists a unique solution $H (\ep) = H (\theta; \ep)$ to (\ref{eq:g}) in the 
		neighborhood $\Bbb_{\de_1}  (\one; \Hcal^s_\crm (S^1))$ of the constant function  $\one$.
	\item[(ii)] The $\Hcal^s_\crm (S^1) $-valued function $\ep \mapsto H (\ep)$ is complex analytic in $ \Dbb_{\de_0}^2 (0)$. The coefficients $H_{j,k} $ of its Taylor series $H(\ep) = \sum\limits_{j,k=0}^{\infty} H_{j,k} \ep_1^j \ep_2^k$ satisfy 
	\begin{gather}
	\text{$ H_{0,0} = \one$, \quad $H_{0,k} = 0$ \ for all $k \in \Nbb$, \quad $H_{1,0} = \tfrac{1}{3}  \cos \te $, 
	} 
	\label{e:Hjk0}
	\\ 
		H_{2,0} = \tfrac{1}{6} G^{-1}  [1+\cos (2 \te)]  , 
	\quad H_{1,1} = \tfrac{1}{6}  G^{-1}[1-\cos (2 \te)] ,   
	\label{e:Hjk1}
	\\
	\ \text{ and } \quad		(H(\ep)|\one) 
		= 
		1+ \tfrac{1}{6} \ep_1^2 + \tfrac{1}{6} \ep_1 \ep_2 + o (|\ep|^2) 
		\quad 
		\text{as} 
		\quad
		\ep \to 0.
	\label{e:<H1>}
	\end{gather}
\end{itemize}
\end{proposition}

\begin{proof}
\emph{Step 1.} We prove statement (i) and the analyticity  part of statement (ii).
Consider the map 
$
f_2 \colon \Hcal^s_\crm (S^1) \times \Cbb^2 
	\rightarrow \Hcal^{s-3}_\crm (S^1) 
	$
defined by 	\[
	f_2 (H;\ep)
	:=
	\left(H - \tfrac{\ep_1 \cos \theta }{3}  H^3\right)
	+
	H^3 \Bo [H]
	+
	\ep_2 \sin(\theta) H^3 H^\prime
	-
	\one .
\]
We observe that $f_2$ is locally bounded, $f_2 $ is analytic w.r.t. $H$ and w.r.t. $\ep$, and that $f_2 (\one;0) = 0$. Since $\Bo [\one] = 0$, the  Fréchet derivative of $f_2$ w.r.t. $H$ at  $(H_0;\ep)=(\one;0)$ is given by the linear homeomorphism $G$, see \eqref{eq:def_G}.
Thus, the analytic implicit function theorem (see \cite[Appendix B]{PT86}) 
proves statement (i) of the proposition and the analyticity statement in (ii). 
	
	
	
\emph{Step 2.} We prove now formulae  \eqref{e:Hjk0}--\eqref{e:<H1>}.
The equation \eqref{eq:g} can be rearranged as 
\begin{equation} \label{e:H-1}
	H - \one =  H^3 \left(\ep_1 \tfrac{1}{3} \cos(\theta) - \ep_2 \sin(\theta)  H' - \Bo [H] \right) ,
\end{equation}
where as before $H' = \pa_\theta H$.
Inserting the series $H (\ep) = \sum_{j,k=0}^{\infty} H_{j,k} \ep_1^j  \ep_2^k$ into \eqref{e:H-1} and using the facts that $H_{0,0} = \one$ and $0 = H'_{0,0} = \Bo[H_{0,0}]$, we get
\begin{multline*}
 H_{1,0} \ep_1 + H_{0,1} \ep_2 + H_{2,0} \ep_1^2  + H_{1,1} \ep_1 \ep_2  + H_{0,2} \ep_2^2 + o(|\ep|^2)  
	\\ 
	 = 
	\left(  \one +  H_{1,0} \ep_1 + H_{0,1} \ep_2 + o(|\ep|) \right)^3 R (\ep),
\end{multline*}
where 
\begin{multline*}
	R (\ep) = 
	\tfrac{1}{3} \cos(\theta) \ep_1 
	-  \sin(\theta)  \ep_2 \bigl[ H'_{1,0} \ep_1 + H'_{0,1} \ep_2 + o(|\ep|) \bigr]  -  \Bo [H_{1,0}] \ep_1 
	\\ 
	 	+ \Bo [H_{0,1}] \ep_2 + \Bo [H_{2,0}] \ep_1^2  + \Bo [H_{1,1}] \ep_1 \ep_2  + \Bo [H_{0,2}] \ep_2^2 + o(|\ep|^2)   .
\end{multline*}
Using $\Bo [\cos \theta] = 0$, the definition  \eqref{eq:def_G} of $G$, and the fact that $\Bo +I$ is invertible, one obtains
\begin{align*}
	H_{1,0} 
	&=  
	\tfrac{1}{3}\cos(\theta) - \Bo [H_{1,0}] =  (\Bo+I)^{-1} \left[\tfrac{1}{3} \cos(\theta) \right] 
	= 
	\tfrac{1}{3} \cos(\te),  
	\\
	H_{0,1} 
	&= 
	- \Bo [H_{0,1}] = (\Bo+I)^{-1} [0] = 0, 
	\\
	H_{2,0} 
	&=  
	3 H_{1,0} \left(\tfrac{1}{3} \cos(\theta) - \Bo [H_{1,0}]\right) - \Bo [H_{2,0}]  = 3 H_{1,0}^2 - \Bo [H_{2,0}] = G^{-1}  [3 H_{1,0}^2]  , 
	\\
	H_{1,1} 
	&=  
	- \sin(\theta) H_{1,0}' - \Bo[H_{1,1}] 
	= G^{-1}[ - \sin (\theta) H_{1,0}'] 
	= G^{-1}\left[\tfrac{1}{3} \sin^2 (\theta)\right], 
	\\
	H_{0,2} 
	& =  
	- \Bo [H_{0,2}] 
	= 
	0. 
\end{align*}
Inductively we get  $H_{0,k} = - \Bo [H_{0,k}] = 0$ for all $k \ge 1$. This completes the proof of \eqref{e:Hjk0}--\eqref{e:Hjk1} and, in turn, implies \eqref{e:<H1>}.
\end{proof}

\begin{remark}\label{r:real_eps}
	If $\ep\in \Rbb^2 \cap \Dbb_{\de_1}^2 (0)$ in the settings of Proposition \ref{p:CsolutionODE}, 
	then $H(\cdot;\ep)$ is real-valued. This follows from the local uniqueness of $H$ and the invariance of equation \eqref{eq:g} under the complex conjugation.
	This implies Theorem \ref{thm:solution_ODE}.
\end{remark}

\section{Hopf-bifurcation curves for the rimming-flow equation\label{s:rfHcurves}}

\subsection{The reduced rimming-flow equation in the abstract settings}
\label{s:ReductionProof}

In this subsection we show that the parametrized dynamical system \eqref{e:PDEHyp2} fits 
into the settings of Section \ref{s:mpHopfB} on abstract multi-parameter Hopf-bifurcations.

In formulae \eqref{eq:f}-\eqref{eq:f1},  the rimming-flow equation \eqref{eq:PDE2} is written in the form 
$\pa_t h  = f(h;\ep)$ with  
real-valued $h$. The abstract settings of Section \ref{s:mpHopfB} require  the complexification of the Fréchet derivative of $f$ in \eqref{e:A}. 

Instead of the use of the abstract complexification, one can employ its natural version just assuming that the function $h$ is 
complex-valued. Let us consider this process in detail.
We introduce the continuous mapping  $f_\crm: H_\crm^4 (S^1) \times \Cbb^2 \to  L^2_\crm(S^1) $ defined for complex-valued functions by the same differential expression as $f$, and consider the associated dynamical system in  the complex space $L^2_\crm(S^1)$. 
Taking $s=4$ and taking $\ep$ in the polydisc $\ep \in \Dbb_{\de_0}^2 (0)$ of Proposition \ref{p:CsolutionODE}, we use the complex analytic parametrization $\ep \mapsto H(\ep) = H (\cdot,\ep) $ from this proposition.
The complex hyperplane  $ \Hyp_\crm (C(\ep)) = \{ h \in L_\crm^2(S^1): \ ( h | \one) = C(\ep) \}$ containing $H(\ep)$ is invariant for $\pa_t h  = f(h;\ep)$, which can be shown in the same way as for the real case in Section \ref{s:MainRes}.
Here $C(\ep) = \int_{S^1} H(\theta,\ep) \frac{\dd \theta}{2\pi}$ can be nonreal for $\ep \not \in \RR^2$.
The equation $\pa_t h  = f_\crm (h;\ep)$ in  the invariant hyperplane  $ \Hyp_\crm (C(\ep))$ 
can be written w.r.t. 
$
u(t,\theta;\ep) := h(t,\theta;\ep)- H (\theta;\ep) \in \dotL^2_\crm(S^1)
$
 as the equation $  \pa_t u =  \wt F (u ; \ep) $ with $\wt F (w ;\ep) := f (w+H(\ep);\ep)$.

The complex Fréchet derivative $D_w \wt F(\zero;\ep) \in \Lcal \bigl(\dotH^4_\crm(S^1), \dotL^2_\crm(S^1) \bigr)$
 w.r.t. the first variable $w$ is given by the operator 
$D_w \wt F(\zero;\ep) = - \pa_\theta Q_\ep[v]$,
where the differential expression $Q_\ep$ is given by the formula \eqref{e:A1}.

By direct computations one can see that, for $\ep \in \Dbb_{\de_0}^2 (0)$,   there exists the complex derivative $D_\ep \wt F(\zero;\ep)$ w.r.t. the 2-parameter $\ep$.
Since the map $\wt F \colon \dotH_\crm^4 (S^1) \times \Dbb_{\de_0}^2 (0) \to  \dotL^2_\crm(S^1) $ is locally bounded, \cite[Theorem A.1]{PT86} implies that 
$\wt F$ is holomorphic in $\dotH_\crm^4 (S^1) \times \Dbb_{\de_0}^2 (0)$. 

Consequently, 
the map $F \colon \dotH^4 (S^1) \times \Bbb_{\de_0}^2 (0;\RR^2) \to \dotL^2 (S^1) $ defined in \eqref{e:PDEHyp2} is a real analytic map. As a complexification $A(\zero;\ep)$ of its Fréchet derivative $D_w F(\zero;\ep) $ one can take $D_w \wt F(\zero;\ep)$.
Moreover, the dynamical system $\pa_t u  = F(u;\ep)$ satisfies the assumptions 
\eqref{h:HComp} and \eqref{h:HAn} of Section \ref{s:mpHopfB}. 

\begin{remark} \label{e:Adef}
Thus, $A(\zero;\ep) [\cdot]$ is a bounded linear operator from $ \dotH_\crm^4 (S^1)$ to $\dotL^2_\crm(S^1)$ defined by 
\begin{gather} \label{e:A=paQ}
A (\zero;\ep) v = A(\zero;\ep)[v] = - \pa_\theta Q_\ep[v]
\end{gather}
with $Q_\ep$ given by \eqref{e:A1}.
However, for the calculation of perturbations of eigenvalues  we use the spectral analysis convention (see, e.g., \cite{Kato}),
which considers $A(\zero;\ep)[\cdot]$ as an unbounded operator in the space $\dotL^2_\crm(S^1)$ with the dense domain $\dom A(\zero;\ep) = \dotH^4_\crm (S^1)$.
\end{remark}

Recall that $\Bo [v] =  b(\pa_\te v+\pa_\te^3 v)$. We apply to $A(\zero;\ep)$ the perturbation results of Section \ref{s:Hypersurf} and of Appendix \ref{a:A} in a vicinity of the unperturbed  2-parameter $\ep^0 = 0$. Note that, in this case 
\begin{equation} \label{e:A000}
\text{$X^\loc (\ep) \equiv \zero$ \qquad and \qquad $A^\loc (\ep) \equiv A(\zero;\ep)$.}
\end{equation}
The unperturbed operator 
\[	A^\loc (0) [v]  = A(\zero;0) [v]
	= -\pa_\th (v+\Bo [v]) = - \pa_\theta v - b  (\pa_\theta^2 v  + \pa_\theta^4 v ) 
\]
is normal and diagonal w.r.t. the basis $\{e^{\ii n \theta}\}_{n \in \Zbb\setminus\{0\}}$ of $\dotL^2_\crm (S^1)$.  More precisely, 
\begin{equation} \label{e:A00}
	A^\loc (0)  \,e^{\ii n \te} \  = \ \om_n  \, \ee^{\ii n \te} 
	\quad 
	\text{with} 
	\
	\om_n = - b(n^4 - n^2) - \ii n,
	\ n \in \Zbb\setminus\{0\}. 
\end{equation}

\begin{proposition} \label{p:Av}
 There exist  constants $\de>0$, $\de_5>0$, and there exist  two holomorphic functions $\la^\pm \colon \Ncal_\crm^{0}  \to \Cbb$, where  $\Ncal_\crm^{0} = \Dbb^2_\de (0)$  is a polydisc centered at $\ep^0=0$, 
with the following properties: 
\begin{itemize}
	\item[(i)] In the polydisc $\Ncal_\crm^{0}$, $A^\loc (\cdot)$ is a holomorphic family of type ($\Acal$) in the sense of \cite{Kato,B85} (see also Appendix \ref{a:A}),  and $A^\loc (\ep)$ is $\Hsec$-sectorial for each $\ep \in \Ncal_\crm^{0}$.
	\item[(ii)]  For all $\ep $ in the $\Rbb^2$-neighborhood $\Ncal^{0}  := (-\de,\de)^2$ of $0$, the part \[ \{\la \in \si (A^\loc (\ep)) \ : \ - \de_5  < \Re \la \} \] of the spectrum of $A^\loc (\ep)$ 
		consists of exactly two complex conjugate simple isolated eigenvalues $\la^\pm (\ep)$; in particular, 
	$\la^\pm (0) = \pm \ii$ \ and \ $\ee^{\mp \ii \theta}$ \ are the associated eigenfunctions of $A^\loc (0)$.
	Moreover, $(\zero;0)$ is a $\si(A)$-critical equilibrium point of $\pa_t u = F(u;\ep)$. 
	\item[(iii)] There exists a holomorphic function $\psi \colon \Ncal_\crm^{0} \to \dotH^4 (S^1)$
	such that 
	\[
	\text{$A^\loc (\ep) \psi (\ep)  = \la^+(\ep) \psi (\ep)$ and 
	$(\psi (\ep)  | e^{-\ii \theta}) = 1$ for all $\ep \in  \Ncal_\crm^{0}$.} 
	\]
	\item[(iv)] The function $\rla(\ep) = \Re \la^+ (\ep)$ is analytic in $\Ncal^{0}$. For $\ep \in \Ncal^{0} $, the equilibrium point 
	$(\zero;\ep)$ of  $\pa_t u =  F (u;\ep)$ is  $\si(A)$-critical exactly when $\rla(\ep) = 0$.
\end{itemize}
\end{proposition}

\begin{proof}
Since $\wt F$ is holomorphic in $\dotH_\crm^4 (S^1) \times \Dbb_{\de_0}^2 (0)$ and 
$A^\loc (\ep) = D_w \wt F(\zero;\ep)$, we see that $A^\loc (\cdot) v$ is a holomorphic 
$\dotL^2_\crm(S^1)$-valued function in $\Dbb_{\de_0}^2 (0)$ for every $v \in \dotH_\crm^4 (S^1)$. In order to prove that $A^\loc (\cdot)$ is a holomorphic family of type ($\Acal$) in $\Dbb^2_\de (0)$ with a sufficiently small $\de$, it is enough to check that $A^\loc (\ep)$ is a closed operator in 
$\dotL^2_\crm(S^1)$ for $\ep \in \Dbb^2_\de (0)$. This follows from \cite[Theorem IV.1.1]{Kato}. 

Indeed, the operator $A^\loc (0)$ is normal and closed in $\dotL^2_\crm(S^1)$. Proposition \ref{p:CsolutionODE} combined with \eqref{e:A=paQ} and \eqref{e:A1} implies that the operator $A^\loc (\ep) - A^\loc (0)$ is $A^\loc (0)$-bounded for small enough $\ep$. Thus, \cite[Theorem IV.1.1]{Kato} implies that $A^\loc (\ep)$ is closed.
 
The representation (\ref{e:A00}) of  $A^\loc (0) $ implies that $A^\loc (0)$ is $\Hsec$-sectorial in $\dotL^2_\crm(S^1)$. 
Combining this with the abstract results of Section \ref{s:Hypersurf} and Appendix \ref{a:A}, one obtains the $\Hsec$-sectoriality in statement (i), possibly, taking smaller $\de>0$.

 Statement (ii) follows from the combination of the unperturbed case (\ref{e:A00}) with Propositions  \ref{p:PerSectC} and the fact that for real 2-parameters $\ep$ the eigenvalues of $A^\loc (\ep) $ appear in complex conjugate pairs. Since $\la^\pm (0) = \pm \ii$ and the other eigenvalues $\la$ of  $A^\loc (0)$
satisfy $\Re \la \le -12 b$, we see that $(\zero,0)$ is a $\si (A)$-critical equilibrium point of $\pa_t u =  F (u;\ep)$. 

Applying Proposition \ref{p:SpCrManif1} (ii) and restricting, if necessary, in this process the $\Rbb^2$-neighborhood $\Ncal^{0}$, we obtain statement (iv). Statement (iii) follows from Proposition \ref{p:PerSimpleC}.
\end{proof}

\subsection{Poof of Theorem \ref{t:AHyp}
\label{ss:PertLa}}

 For small variations of $\ep = (\ep_1,\ep_2)$ near $\ep^0 = 0$, we calculate leading coefficients of the Taylor series 
\begin{equation} \label{e:TayLaPsi}
	\la^+(\ep) 
	=
	\sum_{j,k=0}^\infty \la^+_{j,k} \ep_1^j \ep_2^k 
	\qquad
	\text{and}
	\qquad
	\psi(\ep)
	= \sum_{j,k=0}^\infty \psi_{j,k} \ep_1^j \ep_2^k  
\end{equation}
for the perturbations $\la^+(\ep) $ of the eigenvalue $\la^+_0  = \ii$ of $A^\loc (0)$ 
and for the associated with $\la^+(\ep) $ eigenfunction $\psi (\ep): \theta \mapsto \psi (\ep)(\theta)$ defined by Proposition \ref{p:Av}. 
Our aim is to show that $\rla (\ep) = \Re \la^+(\ep) $ has the properties described in Lemma \ref{l:n=2Det<0}, 
and to prove in this way Theorem \ref{t:AHyp}.

\begin{remark}
Note that the calculations of $\la^- (\ep)$ and of an eigenfunction of $A^\loc (\ep)$ associated with $\la^- (\ep) $ is not needed since $\la^- (\ep) = \overline{\la^+ (\ep)} $.
\end{remark}

The operator $A^\loc (\ep) = A (\zero;\ep) $ is given by \eqref{e:A=paQ} and \eqref{e:A1}.  In the unperturbed case, the eigenvalues and the eigenfunction of $A^\loc (0)$ are given by \eqref{e:A00}. So, $\la^+ (0)  = \la^+_{0,0} = \ii$  and $\psi_{0,0} (\te) = e^{-\ii \te}$ are the corresponding eigenvalue and the eigenfunction.
Note that $\psi_{0,0}$ satisfies the normalization $(\psi_{0,0}   | e^{-\ii \theta}) = 1$, which was used in Proposition \ref{p:Av} to single out the eigenfunction.

In order to determine the subsequent coefficients of the Taylor series \eqref{e:TayLaPsi}, we test the corresponding eigenvalue problem $\lambda^+(\ep) \psi(\ep) = A^\loc (\ep)[\psi(\ep)]$ with suitable test functions $\vphi \in \dotH^1_\crm(S^1)$, i.e. we consider the equation
\begin{equation} \label{eq:psiphi}
	(\la^+(\ep) \psi(\ep)| \vphi ) 
	= 
	(A^\loc (\ep) [\psi(\ep)] | \vphi)  
	= 
	(Q_\ep[\psi(\ep)] | \pa_\theta \vphi),
	\quad 
	\vphi \in \dotH^1_\crm(S^1).
\end{equation}

\begin{remark} 
In this subsection the linear operator $\Bo: v \mapsto  b(\pa_\te v+\pa_\te^3 v)$ is perceived as an unbounded operator in the space $\dotL^2_\crm(S^1)$ with the dense domain $\dom A(0;\ep) = \dotH^3_\crm (S^1)$.
\end{remark}

Using the leading terms of the expansion of $H(\ep)$ of Proposition \ref{p:CsolutionODE} and observing that 
\begin{equation*}
	0 
	= 
	\Bo[\psi_{0,0}] 
	= 
	\Bo[H_{1,0}] 
	= 
	\Bo[H_{0,0}]  
	= 
	\pa_\theta H_{0,0} 
	= 
	H_{0,1},
\end{equation*}
we calculate some of the leading terms in the Taylor expansion of $Q_\ep[\psi(\ep)]$:
\begin{equation*}
\begin{split}
	& Q_\ep[\psi(\ep)]  
		=  
	\psi(\ep)  
	+ 
	\bigl(1+\ep_1 \tfrac{\cos\te}{3} + \ep_1^2 H_{2,0} + \ep_1 \ep_2 H_{1,1} + \dots\bigr)^3 
	\times \\
	& \qquad \qquad \qquad \qquad \qquad \qquad \qquad \qquad \qquad \qquad
	 \times  \bigl(\ep_1 \Bo[\psi_{1,0} ] + \ep_2 \Bo[\psi_{0,1} ] + \dots\bigr)
	\\ 
	& \qquad \qquad \quad
	 - \ep_1 \cos(\te) 
	\bigl(1+\ep_1 \tfrac{\cos\te}{3} + \ep_1^2 H_{2,0} + \ep_1 \ep_2 H_{1,1} + \dots\bigr)^2  
	\times \\
	& \qquad \qquad \qquad \qquad \qquad \qquad \qquad \qquad \qquad \qquad
	 \times 
	\bigl(e^{-\ii \te} + \ep_1 \psi_{1,0}  + \ep_2 \psi_{0,1}  + \dots\bigr) 
	\\ & \qquad \quad 
	+ 3 \bigl(1+\ep_1 \tfrac{\cos\te}{3} + \ep_1^2 H_{2,0} + \dots\bigr)^2  
	\bigl(\ep_1^2 \Bo[H_{2,0}] + \ep_1 \ep_2 \Bo[H_{1,1}] + \dots\bigr) 
	\times \\
	& \qquad \qquad \qquad \qquad \qquad \qquad \qquad \qquad \qquad \qquad  \qquad \qquad
	 \times 
	\bigl(e^{-\ii \te} + \ep_1 \psi_{1,0}  + \dots\bigr)
	\\ & \qquad \
	+ \ep_2 \sin(\te) 
	\bigl(1+\ep_1 \tfrac{\cos\te}{3} + \ep_1^2 H_{2,0} + \dots\bigr)^2 
	\times
	\\ &   \qquad \qquad \qquad 
	\times 
	\Bigl[ 
	3\bigl(-\ep_1 \tfrac{\sin\te}{3} + \ep_1^2 H'_{2,0} + \dots\bigr) 
	\bigl(e^{-\ii \te} + \dots\bigr)  \Bigr. \\
	& \qquad \qquad \qquad \qquad \qquad 
	+ \Bigl.
	\bigl(1+\ep_1 \tfrac{\cos\te}{3} + \dots\bigr) 
	\bigl(-\ii e^{-\ii \te} + \ep_1 \psi_{1,0}' + \ep_2 \psi_{0,1}' + \dots\bigr) 
	\Bigr],
\end{split}
\end{equation*}
where $f'$ means $\pa_\theta f$ and  we denote by $\dots$ unnecessary higher order terms.
Moreover, we observe that 
the normalization $(\psi(\ep) | e^{-\ii \theta}) = (\psi(\ep) | \psi_{0,0} ) = 1$ implies $
	(\psi_{j,k}  | \psi_{0,0} ) 
	= 0$, \ $ 	(j,k) \in \Nbb_0^2\setminus(0,0)$.
In addition, we have that
\[
\text{
$
	\pa_\te \psi_{0,0}  
	= 
	-\ii \psi_{0,0}  
	= 
	-\ii \ee^{-\ii \te}, $ 	\
$	\Bo^*[e^{-\ii \te}] 
	= 
	-\Bo[\ee^{-\ii \te}] 
	= 
	0$, 
	}
	\]
	and hence, $	(\Bo[\psi_{j,k} ] | \ee^{-\ii \th})  =  	0$.
Now we consider the eigenvalue problem \eqref{eq:psiphi} with the test function $\vphi = \ii \psi_{0,0}  = \ii e^{- \ii \te}$. This leads to 
\begin{align*}
	- \ii \la^+(\ep)  
	& =  
	(\la^+(\ep) \psi(\ep)\, | \, \ii \psi_{0,0} ) 
	= 
	(Q_\ep[\psi(\ep)] \,  | \, \ii \pa_\te \psi_{0,0} ) 
	=  
	(Q_\ep[\psi(\ep)]\, | \, \ee^{-\ii \te}) 
	\\  
	& =  
	1 
	+ 
	\left(\ \ep_1^2 
	\left(\cos(\te) \Bo[\psi_{1,0} ] \, | \, \ee^{-\ii \te} \right) 
	+ 
	\ep_1 \ep_2 
	\left(\cos(\te) \Bo[\psi_{0,1} ]\, | \, \ee^{-\ii \te}\right) + \dots 
	\right) 
	\\ & \quad  
	- \left(\ \ep_1 \left( \cos(\te) [\ep_1 \psi_{1,0}  + \ep_2 \psi_{0,1} ] \, | \, \ee^{-\ii \te} \right) 
	+ \tfrac23 \ep_1^2 \left( \cos^2(\te) e^{-\ii \te} \, | \, e^{-\ii \te} \right)  
	+ \dots 
	\right)
	\\ & \quad  
	+ 3 \left( e^{-\ii \te} \Bo \left[ \ep_1^2  H_{2,0} + \ep_1 \ep_2 H_{1,1} + \dots 
	\right] 
	\, \big\vert \, \ee^{-\ii \te} \right) 
	\\ & \quad  
	-  
	\left(\ep_1 \ep_2 \, (\sin^2(\te) \ee^{-\ii \th} \, | \, \ee^{-\ii \te}) - 
	\ep_2 \, (\, \sin(\te) [\ep_1 \psi_{1,0}' + \ep_2 \psi_{0,1}'] \, | \, \ee^{-\ii \te} \, ) + \dots 
	 \right)
\end{align*} 
Thus, in addition to the starting point 
$\la^+_{0,0} = \ii$ of our calculations, we have
$
\la^+_{1,0} = \la^+_{0,1} = 0.
$ 
Taking into account that
\begin{gather*}
	(\cos(\te) \Bo[v] | \ee^{-\ii \te} ) 
	= 
	(v | -\Bo[\cos(\te) \ee^{-\ii \te}]) 
	= 
	3 \ii b (v | e^{-2\ii \th}), 
	\\
	(\cos(\te) \psi_{j,k}  | \ee^{-\ii \te}) 
	= 
	(\psi_{j,k}  | \cos(\te) \ee^{-\ii \te}) 
	= 
 	\tfrac{1}{2} (\psi_{j,k}  | \ee^{-2\ii \te}), 
 	\\
    (\cos^2(\te) e^{-\ii \te} | e^{-\ii \te}) 
    = 
    (\sin^2(\te) \ee^{- \ii \th} | \ee^{-\ii \te}) 
    = \tfrac{1}{2}, 
    \\
  	(e^{-\ii \te} \Bo[v] | \ee^{-\ii \te}) 
  	=  
  	(v | -\Bo[\one]) 
  	= 
  	0, 
  	\\  
  	(\sin(\te) [\psi_{j,k} ]' | \ee^{-\ii \te}) 
  	=  
  	([\psi_{j,k} ]' | -\tfrac{1}{2 \ii} \ee^{-2\ii \te}) 
  	= 
   	-(\psi_{j,k}  | \ee^{-2\ii \te}),
\end{gather*}
we further obtain 
\begin{gather*}
	\la^+_{2,0} 
	=  
	\ii \Bigl[\bigl(3\ii b - \tfrac{1}{2}\bigr)  (\psi_{1,0}  | e^{-2\ii \th}) 
	- \tfrac{1}{3}\Bigr] 
	= 
	\bigl(-3b - \tfrac{\ii}{2}\bigr)  
	(\psi_{1,0}  | e^{-2\ii \th}) 
	- \tfrac{\ii}{3}, 
	\\
	\la^+_{1,1} 
	=    
	- \ii 
	(\psi_{1,0}  | e^{-2\ii \th})  
	+ 
	\bigl(-3b - \tfrac{\ii}{2}\bigr)  
	(\psi_{0,1}  | e^{-2\ii \th})  
	- \tfrac{\ii}{2},
	\qquad 
	\la^+_{0,2} 
	= 
	- (\psi_{0,1}  | e^{-2\ii \th}).
\end{gather*}

To obtain the second-order Taylor coefficients for $\la^+ (\cdot)$,  it remains to determine the inner products $(\psi_{1,0}  | e^{-2\ii \th})$ and $(\psi_{0,1}  | e^{-2\ii \th})$. For this purpose, we use $\vphi = (-2\ii)^{-1} e^{-2\ii \th}$ as a test function in \eqref{eq:psiphi}. This yields
\begin{align*} 
	& (Q_\ep[\psi(\ep)] | e^{- 2\ii \te}) 
	 = \left( A^\loc (\ep) [\psi(\ep)]  \mid (-2\ii)^{-1} e^{-2\ii \th} \right)  = 
	\left( \la^+(\ep) \psi(\ep) \mid (-2\ii)^{-1} e^{-2\ii \th} \right)   
	\\ 
	& \qquad \qquad = 
	\tfrac{1}{2\ii} \Bigl( \bigl[\ii + \ep_1^2 \la^+_{2,0} + O(|\ep|^2) \bigr] \left[e^{-\ii \th} + \ep_1 \psi_{1,0}  + \ep_2 \psi_{0,1}  + O(|\ep|^2) \right] \, \Big\vert \, 		e^{-2\ii \th} \Bigr) 
	\\ 
	&  \qquad \ = 
	\tfrac{1}{2} \ep_1 
	(\psi_{1,0}  | e^{-2\ii \te} ) + \tfrac{1}{2} \ep_2 (\psi_{0,1}   |  e^{-2\ii \te} ) + O(|\ep|^2) ; 
\end{align*}
on the other hand, using the expansion of $Q_\ep[\psi(\ep)]$, we find that
\begin{align*} 
	(Q_\ep [\psi (\ep)] \mid \ee^{- 2\ii \te}) 
	& =
	\ep_1 ( \psi_{1,0} \mid e^{-2\ii \te} ) + \ep_2 ( \psi_{0,1}  \mid e^{-2\ii \te} ) + O(|\ep|^2)
	\\ 
	& \qquad \qquad 
	+ \Bigl( \Bo \left[\ep_1  \psi_{1,0}  + \ep_2  \psi_{0,1}  + O(|\ep|^2)\right] \, \Big\vert \, e^{-2\ii \te} \Bigr)
	\\
	& \qquad \qquad 
	- \ep_1 ( \cos (\te) e^{- \ii \te} \mid  e^{-2\ii \te} ) 
	\\ 
	& \qquad \qquad 
	+ \ep_2 \ ( \, - \ii \sin (\te)  e^{-\ii \te} \mid e^{-2\ii \te} \, ) +  O(|\ep|^2)
	\\ 
	& = 
	\ep_1 (1+6 \ii b)( \psi_{1,0} \mid e^{-2\ii \te} ) 
	+ 
	(1+6 \ii b) \ep_2 
	(\psi_{0,1}  \mid e^{-2\ii \te}) 
	\\
	& \qquad \qquad 
	- 
	\tfrac{\ep_1}{2} 
	+ 
	\tfrac{\ep_2}{2} 
	+ O(|\ep|^2) .
\end{align*}
%
This implies
$
	(\psi_{1,0} \mid e^{-2\ii \te} )  
	= 
	\frac{1}{
	(1+12 \ii b)}
	= 
	\frac{1-12\ii b}{1+144b^2}
	$
	  and  
	$ 
	(\psi_{0,1}  \mid e^{-2\ii \te} ) 
	= 
	\frac{-1+12\ii  b}{1+144b^2} $.
Finally, 
\begin{align*} 
	\la^+_{2,0} 
	&=  
	\frac{(-6b - \ii)(1- 12 \ii b)}{2(1+144b^2)} 
	- 
	\frac{\ii}{3} 
	= 
	\frac{-18b+\ii (72b^2-1)}{2(1+144b^2)} 
	- 
	\frac\ii3, 
 	\\
	\la^+_{1,1} 
	&=  
	- \ii \frac{1-12 \ii b}{1+144b^2} +  \frac{(-6b - \ii)(-1+12\ii  b)}{2(1+144b^2)}  - \frac\ii2 
	=
	\frac{-6b -\ii (1+72b^2)}{2(1+144b^2)}-\frac\ii2,
	\\
	\la^+_{0,2} 
	&= 
	\frac{1-12\ii  b}{1+144b^2}  .
\end{align*}
Since $\rla (\ep) = \Re \la^+(\ep)$, we observe that the above calculations yield 
\begin{align} \label{e:rela}
	(1+144b^2) \rla (\ep) =  -9b \ep_1^2 -3b \ep_1 \ep_2 +
	\ep_2^2 + O (|\ep|^3) .
\end{align}
Lemma \ref{l:n=2Det<0} combined with Proposition \ref{p:Av} completes the proof of 
Theorem \ref{t:AHyp}.

\subsection{Proof of Theorem \ref{t:HB}}
\label{s:Stability}

\begin{remark}
Statement (iii) of Theorem \ref{t:HB} is a simplified form of the Hopf bifurcation theorem and follows from the combination of Lemma \ref{l:n=2Det<0}, Theorem \ref{t:n=2Det<0},
and Proposition \ref{p:HbOnPath} with the explicit form  \eqref{e:rela} of the leading Taylor coefficients of the function $\rla$.
More specifically, the statement \eqref{e:HBfTh} follows from Theorem \ref{t:n=2Det<0}. Statement \eqref{e:periodic} is a part  of the definition of a Hopf bifurcation on a path (see Definitions \ref{d:AHbpath} and \ref{d:Hbpath}) and follows from Proposition \ref{p:HbOnPath}. Thus, it remains to prove statements (i) and (ii) of Theorem \ref{t:HB}.
\end{remark}

The main goal of this subsection is to obtain the stability and instability statements (i)--(ii) of Theorem \ref{t:HB} from  the Principle of Linearized Stability for abstract fully nonlinear parabolic equations, see the monograph of Lunardi  \cite{L12}.

Recall that for fully nonlinear parabolic equations one cannot expect global existence of strong solutions for a general initial state. 
Roughly speaking, the main reasons  for this effect in the case of the rimming flow equation \eqref{eq:f}--\eqref{eq:f1} can be either the loss of the parabolicity if the solution $h$ vanishes at a certain point, or blow-up of the solution.
However,  if the parameter $\ep$ is small enough and  the eigenvalues are in the 'stable complex half-plane' one can 
ensure that these effects do not happen for the initial values $h_0$ that are close enough to the 2-dimensional surface of uniformly positive perturbed equilibria   $H (\ep)$ of Theorem \ref{thm:solution_ODE}, and moreover, the Principle of Linearized Stability can be used to derive the exponential convergence to the equilibrium $H (\ep)$ lying in the same invariant hyperplane as $h_0$.

On the rigorous level, our aim is to show that the reduced rimming flow equation $\pa_t u = F (u;\ep)$  (see \eqref{e:PDEHyp2}) fits
into the settings  of \cite[Section 9.1.1]{L12} near the stationary solution $u \equiv \zero$ for small enough positive  $\ep_1$ and $\ep_2$. 
In the settings of the Principle of Linearized Stability \cite[Theorem 9.1.2]{L12},
the parabolicity assumption corresponds to \cite[assumptions (8.0.3) and (9.1.2)]{L12}.
Let us show how the validity of these  assumptions can be obtained from the perturbative arguments  of Section VII.2.1 of  the monograph of Kato \cite{Kato}.

It follows from \eqref{e:A000} and \eqref{e:A00} that
the graph norm associated with the operator \linebreak $A (\zero,0) \colon  \dotH_\crm^4 (S^1)  \subset \dotL^2_\crm (S^1) \to \dotL^2_\crm (S^1) $  is equivalent to the norm of $\dotH_\crm^4 (S^1) $.
Section \ref{s:ReductionProof} implies that there exists $\de>0$ such that statement (i) of Proposition \ref{p:Av} is fulfilled. Taking smaller $\de >0$ if necessary, one can ensure that the graph norm of $A(\zero;\ep)$ on $\dom A(\zero;\ep) = \dotH^4_\crm(S^1)$ is equivalent to the norm 
 $ \dotH^4_\crm(S^1)$ for all $\ep \in \Dbb^2_\de (0)$, see \cite[Section VII.2.1]{Kato}.
 The $\Hbb$-sectoriality of $A(\zero;\ep)$ was proved in Proposition \ref{p:Av} (i).
 Consequently, \cite[assumptions (8.0.3) and (9.1.2)]{L12} are satisfied for $A(\zero;\ep)$ with $\ep \in \Dbb^2_\de (0)$.

We assume in the rest of this subsection that $\ep \in \Dbb^2_\de (0)$. 
 Since $\wt F \colon \dotH_\crm^4 (S^1) \times \Dbb_{\de_0}^2 (0) \to  \dotL^2_\crm(S^1) $  is holomorphic, one sees that the results of 
\cite[Sections 8.1-8.2]{L12} are applicable to the initial value problem 
\begin{gather} \label{e:IVPtF}
\pa_t u = \wt F (u;\ep), \qquad u (0) = u_0,
\end{gather}
  with
$u_0$ in a certain $\dotH^4_\crm(S^1)$-neighborhood $ \Ncal (\ep)$ of $\zero$.
In particular, for every $u_0  \in \Ncal (\ep)$, there exists $\wt t = \wt t (\ep,u_0) \in (0,+\infty)$ such that  \eqref{e:IVPtF} has a unique solution $u(t)$, $t \in [0, \wt t (\ep,u_0)]$, in the sense of \cite[Theorem 8.1.1]{L12}. Moreover, there exists a maximal interval  $[0, \tau (\ep,u_0))$ with $\tau (\ep,u_0) \in (0,+\infty]$ such that such a unique solution $u(t)$ to  \eqref{e:IVPtF} exists in the sense of \cite[Proposition 8.2.1]{L12} on  all intervals  $[0, \wt t ] \subset [0, \tau (\ep,u_0))$ and the trajectory $\{ u(t) \ : \ t \in [0,  \tau (\ep,u_0))\}$ stays in $\Ncal (\ep)$.

Let us take  $\delta_3,\delta_4 \in (0,\de)$ and the function $\ep_1 \mapsto \Gtwo (\ep_1)$ such that the statements of Theorem \ref{t:AHyp} are satisfied. Strict solutions are understood in the sense of \cite[Definition 4.1.1 and Section 8.1]{L12}.
For a Banach space $\Ycal$, a Hölder exponent  $\alpha \in (0,1)$, and a constant  $T>0$, the weighted Hölder space $C^\alpha_\alpha ((0,T];\Ycal)$
is defined \cite{L12} as the set of bounded functions $f \colon (0,T] \to \Ycal$ with the property 
\[
\sup_{0<\de<T} \left(  \de^\alpha  \sup_{\de\le s<t \le T} \frac{\|f (t) - f(s) \|_\Ycal}{|t-s|^\alpha} \right) \ < \ +\infty .
\]

\begin{proposition}[exponential stability] \label{p:stability}
For all $\ep \in (0,\delta_3)\times (0,\delta_4)$ satisfying
$0 < \ep_2 < \Gtwo (\ep_1) $, there exist constants $\rho_1=\rho_1 (\ep)>0$, $M=M(\ep)>0$,  and $\omega=\omega(\ep) > 0$  such that the following statements hold true for all $u_0 \in \Bbb_{\rho_1} (\zero; \dot{\Hcal}^4_\crm (S^1))$ and all $\alpha \in (0,1)$  with a certain $T = T (\ep,u_0) >0$:
\begin{itemize}
\item[(i)] There exists a unique strict solution 
\begin{equation} \label{e:StrSolClass}
u \   \in  \ C ([0,T]; \dotH^4_\crm (S^1) ) \, \cap \, C^1 ([0,T]; \dotL^2_\crm (S^1) ) \, \cap \, C^\alpha_\alpha ((0,T]; \dotH^4_\crm (S^1)) 
\end{equation}
  to the initial value problem \eqref{e:IVPtF}. If, additionally, the initial value $u_0$ is a real-valued function, the solution 
  $u$ is also real-valued. 
\item[(ii)] The interval $[0, +\infty)$ is the maximal  interval $[0, \tau (\ep,u_0))$ for the continuation of this strict solution $u$ in the sense of \cite[Proposition 8.2.1]{L12}.
\item[(iii)] 
$	\|u(t) \|_{\Hcal_\crm^4 (S^1)} + \|\pa_t u(t) \|_{L^2_\crm (S^1)} 
	\leq
	M e^{-\omega t} \|u_0 \|_{\Hcal_\crm^4 (S^1)}
	$ \quad for all $\ 	t > 0. $
\end{itemize}	
\end{proposition}

\begin{proof}
Let us fix an arbitrary $\al \in (0,1)$. Let $\ep \in  (0,\delta_3)\times (0,\delta_4)$.
Under our assumptions, $\ep \in  (0,\delta_3)\times (0,\delta_4) \subset \Dbb^2_\de (0)$ ensures that $A(\zero;\ep)$ is $\Hbb$-sectorial.
Since the $\Hbb$-sectoriality property is stable in the sense of \cite[Proposition 2.4.2]{L12}, we may choose a radius $\rho_0 = \rho_0 (\ep)>0$ such that \cite[assumptions (8.0.3) and (8.1.1)]{L12} are satisfied for the map $\wt F(u_0;\ep)$ for every $u_0 \in \Bbb_{\rho_0 (\ep)} (\zero; \dot{\Hcal}^4_\crm (S^1))$.
 This allows one to apply \cite[Theorem 8.1.1]{L12}   in order to build a unique local solution satisfying \eqref{e:StrSolClass}
to the initial value problem \eqref{e:IVPtF} for every initial value $u_0 \in \Bbb_{\rho_0 (\ep)} (\zero; \dot{\Hcal}^4_\crm (S^1))$. Note that  \cite[Theorem 8.1.1]{L12}  implies that 
$T = T (\ep,u_0) >0$ can be chosen independently of the choice of  $\al \in (0,1)$.

By Proposition \ref{p:Av}, we have $\sigma (A(\zero;\ep)) \subset \{ \la \in \Cbb \ : \ \re \la \le \re \la^+ (\ep) \}$. 
Assume now additionally that $0 < \ep_2 < \Gtwo (\ep_1) $. Then Theorem \ref{t:AHyp} implies that $\re \la^+ (\ep) < 0$, and so,
$\zero$ is a spectrally exponentially stable steady state of $\pa_t u =  F (u;\ep)$. Thus, the Principle of Linearized Stability in the form of \cite[Theorem 9.1.2]{L12} is applicable
and proves statements  (i)-(iii) in the complex case. 

Since the proposition assumes that $\ep_1$ and $\ep_2$ are real, Remark \ref{r:real_eps} implies that $H (\cdot;\ep)$ is real-valued. Therefore, if the initial value $u_0$ is real-valued, the solution $u$ in the sense of statement (i) remains a solution after taking complex conjugation. The uniqueness of the solution $u$ implies that it is real-valued. This completes the proof. 
\end{proof}

Consider now the case, where $\ep \in (0,\delta_3)\times (0,\delta_4)$ and $ \Gtwo (\ep_1) < \ep_2  $.
Then $\zero$ is an unstable steady state of $\pa_t u = F (u;\ep)$ in the sense described by the next proposition,
where backward mild  solutions are understood in the sense of \cite[Sections 4.1 and 4.4]{L12} (see also the proof of \cite[Theorem 9.1.3]{L12}).

\begin{proposition}[instability] \label{p:instability}
Assume that $\ep \in (0,\delta_3)\times (0,\delta_4)$ and $ \Gtwo (\ep_1) < \ep_2  $.
Then there exists $u_1 = u_1 (\ep) \in \dotH^4 (S^1)$ 
with $\|u_1\|_{ \Hcal^4 (S^1)} >0$ and with the following property:
for arbitrarily  small $\rho>0$ and any $\alpha \in (0,1)$, there exists $T>0$ and an initial state $u_0 \in \Bbb_\rho (\zero; \dotH^4 (S^1))$   
such that the terminal value problem
\begin{gather} \label{e:IVPtFreal}
\qquad \pa_t u = F (u;\ep), \qquad u (T) = u_1,
\end{gather}
has a backward mild  solution $u(\cdot)$ in the Hölder space $C^\alpha ([0,T]; \dotH^4 (S^1))$
 with 
$u(0) = u_0$ \  and \ $\pa_t u \in C^\alpha ([0,T]; \dotL^2 (S^1))$.
\end{proposition}

\begin{proof}
Since $\ep \in (0,\delta_3)\times (0,\delta_4)$ and $ \Gtwo (\ep_1) < \ep_2  $, 
Theorem \ref{t:AHyp} implies that the part of the spectrum of $A(\zero;\ep)$ in $\{\la \in \Cbb \ : \  - \de_5 < \Re \la\}$ consists of exactly two eigenvalues $\la^\pm (\ep)$ with $\re \la^\pm (\ep )> 0$.
The proposition now follows from the proof of \cite[Theorem 9.1.3]{L12} combined with \cite[Proposition 4.4.12]{L12}.
\end{proof}


\appendix
\gdef\thesection{\Alph{section}} 
\makeatletter
\renewcommand\@seccntformat[1]{\appendixname\ \csname the#1\endcsname.\hspace{0.5em}}
\makeatother

\section{Multi-parameter perturbations of isolated \linebreak eigenvalues}\label{a:A}

The multi-parameter perturbations of isolated eigenvalues of linear operators were considered in the supplement to the english 
edition of the monograph of Baumgärtel \cite{B85}. Various special cases  were  studied and used in applied contexts, e.g., in  \cite{G09,K14,AK17}.  Corollary 1 of \cite[Supplement]{B85} is complemented in this appendix with a normalization of the perturbed eigenvector and adapted to the needs of Sections \ref{ss:PertLa} and \ref{s:mpHopfB}. In particular, we provide the necessary statements concerning the $\Hbb$-sectoriality and the complexification  of real analytic operators. 

Let $n \in \Nbb$. Let $\Ycal$ be a complex Banach space with a dual $\Ycal'$. By $\<\cdot, \cdot\>$ we denote the corresponding bilinear pairing of $\Ycal$ and $\Ycal'$. Let $\ep^0$ belong to an open subset $V_\crm$ of $\Cbb^n$. A family
$\{T(\ep)\}_{\ep \in V_\crm}$ of closed in $\Ycal$ linear operators $T(\ep)$ is called \cite{B85,Kato} a holomorphic family of type ($\Acal$) if the following conditions are fulfilled: (i) the domain  $\dom T(\ep) =: \Df$ is independent of $\ep \in V_\crm$, (ii) $T(\ep) u$ is holomorphic in $V_\crm$ for every $u \in  \Df$.

\begin{proposition}[cf. \cite{B85}] \label{p:PerSimpleC}
Let $\{T(\ep)\}_{\ep \in V_\crm}$ be a holomorphic family of type ($\Acal$). 
Let $\la_0$ be a simple isolated eigenvalue of $T (\ep^0)$. Then there exist $\delta_0, \delta_1>0$ and a holomorphic function $\Lambda \colon \Bbb_{\delta_1} (\ep^0; \Cbb^n) \to \Dbb_{\delta_0} (\la_0)$ such that $\Lambda (\ep^0) = \la_0$ and such that for every 
$\ep \in \Bbb_{\delta_1} (\ep^0; \Cbb^n)$ the following statements hold:
\begin{enumerate}
\item[(i)] $ \La (\ep) $ is a simple eigenvalue of $T (\ep)$,
\item[(ii)] $ \La (\ep) $ is the only point of the spectrum $\si (T (\ep))$ in the closed $\Cbb$-disc 
$\overline{\Dbb_{\delta_0} (\la_0)}$.
\end{enumerate}
If, additionally, $\psi^0 \in \ker (T (\ep^0) - \la_0 I_\Ycal) $ and a holomorphic $\Ycal'$-valued function $\psi_* \colon V_\crm \to \Ycal'$ satisfy $\<\psi^0,\psi_* (\ep^0)\> \neq 0$, then there exists $\de_2>0$ and  
 a holomorphic function \linebreak $\psi \colon \Bbb_{\delta_2} (\ep^0; \Cbb^n) \to \Ycal$,
such that additionally to (i)--(ii) the following statement is fulfilled for every 
$\ep \in \Bbb_{\delta_2} (\ep^0; \Cbb^n)$: 
\begin{enumerate}
\item[(iii)] $\psi (\ep) \in \Df$, \qquad $T(\ep) \psi (\ep)  = \La (\ep) \psi (\ep)$, \quad and 
$\<\psi (\ep),\psi_* (\ep)\> = 1$.
\end{enumerate}
\end{proposition}

\begin{proof} Let us consider a $\Cbb$-disc $\Dbb_{\delta_0} (\lambda_0)$ such that $\la_0$ is the only point of the spectrum $\si (T (\ep^0))$ in $\Dbb_{\delta_0} (\lambda_0)$.
Statements (i) and (ii) are proved in Corollary 1 of \cite[Supplement]{B85}, which implies additionally that 
 there exists a holomorphic  function $P \colon \Bbb_{\delta_1} (\ep^0; \Cbb^n) \to \Lcal (\Ycal)$ with the corresponding Riesz projectors 
 \linebreak
 $
P (\ep) := - \frac1{2\pi \ii} \int_{\Tbb_{\delta_0} (\lambda_0)} (T(\ep) - \lambda I_\Ycal)^{-1} \dd \lambda $
as its values. Here   $\Tbb_{\delta_0} (\lambda_0) = \pa \Dbb_{\delta_0} (\lambda_0) $ denotes  the boundary  of the disc $\Dbb_{\delta_0} (\lambda_0) $. 
In order to prove statement (iii), we consider  the holomorphic function $\wt \psi : \Bbb_{\delta_1} (\ep^0; \Cbb^n) \to \Df$ defined by $\wt \psi (\ep) = P(\ep) \psi^0$. Then $\wt \psi (\ep^0) = \psi^0$,  and so there exists a neighborhood of $\ep^0$, where $a (\ep) =  \< \wt \psi (\ep),\psi_* (\ep) \> $ is a holomorphic function without zeroes. Thus, the function $\psi =\tfrac{1}{a} \wt \psi$ satisfies (iii).
\end{proof}

\begin{proposition} \label{p:PerSectC}
Let $\{T(\ep)\}_{\ep \in V_\crm}$ be a holomorphic family of type ($\Acal$). 
Assume that $T (\ep^0)$ is an $\Hsec$-sectorial operator such that, for a certain $\zeta_1 \in \Rbb$, the part of its spectrum in the half-plane 
$\hf_1 := \{ \zeta \in \Cbb : \Re \zeta > \zeta_1\}$ consists of at most a finite number $m \in \{0\} \cap \Nbb$ of simple eigenvalues $\la_j$, $1 \le j \le m$. Then there exist constants $\delta_0$, $\delta_1$, $\delta_2$, $\omega $, $\zeta_0 >0$ and
 holomorphic functions \ $\La_j : \Bbb_{\delta_1} (\ep^0; \Cbb^n) \to \Dbb_{\delta_0} (\la_j)$, \quad $1 \le j \le m$, 
such that:
\begin{itemize}
\item[(i)] $\{T(\ep)\}_{\ep \in V_\crm}$ is uniformly $\Hsec$-sectorial locally near $\ep^0$ in the sense that 
$T(\ep) \in \Hsec (\omega, \zeta_0) $ for all $\ep \in \Bbb_{\delta_1} (\ep^0; \Cbb^n)$ (see Section \ref{s:Hypersurf} for the definition of $\Hsec (\omega, \zeta_0)$).
\item[(ii)] $\La_j (\ep^0) = \la_j$ and $ \La_j (\ep) $ is a simple eigenvalue of $T (\ep)$ for every $\ep \in \Bbb_{\delta_1} (\ep^0; \Cbb^n)$ and $1 \le j \le m$.
\item[(iii)] $ \{ \zeta \in \si (T (\ep)) : \Re \zeta > \zeta_1+\delta_2\} = \{\La_j (\ep)\}_{j=1}^m$  for all $\ep \in \Bbb_{\delta_1} (\ep^0; \Cbb^n)$.
\end{itemize}
\end{proposition}

\begin{proof} 
(i) It can be shown similarly to the 1-parameter case of \cite[Section VII.2.1]{Kato} that for a holomorphic family $\{T(\ep)\}_{\ep \in V_\crm}$ of type ($\Acal$), for any compact subset $V_{\crm,0}$ of $V_{\crm}$, and for any $\de>0$, there exists $\delta_1 >0$ such that the uniform estimate 
\begin{gather} \label{e:Te1-Te2}
\|T(\ep^1) u - T(\ep^2) u \|_{\Ycal} \le \de (\|u\|_{\Ycal}+\| T (\ep) u \|_\Ycal) , \quad \ep^1,\ep^2,\ep \in V_{\crm,0}, \quad u \in \Df ,
\end{gather}
holds whenever $|\ep^1 -\ep^2| < \de_1$.
The arguments of the proofs of \cite[Theorems IX.2.4 and IX.2.6]{Kato} 
imply now statement (i).

(ii) Changing $\de_1$ to a smaller positive value if necessary, we can ensure that it satisfies Proposition \ref{p:PerSimpleC} for each $j \le m$. This proves (ii).

(iii) Let us take $\de_1$, $\de_2>0$ so small that (i) and (ii) are valid for $\de_1$ and  $\{\la_j\}_{1\le j\le m}$ is a subset of a slightly shifted open half-plane $\hf_2:=\{ \zeta \in \Cbb : \Re \zeta > \zeta_1+\delta_2\} $. Then $\{\la_j\}_{1\le j\le m}$ is also a subset of the triangular-shaped set
$\Tcal := \hf_2 \cap (\Cbb \setminus (\xi_0 +\Si_{\pi/2+\om})) $
with the boundary $\pa \Tcal \subset \rho (T(\ep^0))$.
Combining (\ref{e:Te1-Te2}) with \cite[Theorems IV.2.14 and IV.3.16]{Kato} and applying these results to $T(\ep)$, one sees that for $\ep \in \Bbb_{\delta_1} (\ep^0; \Cbb^n)$ with a sufficiently small $\delta_1>0$ the boundary $\pa \Tcal$ separates the 
spectra of  all operators $T(\ep)$ with $\ep \in \Bbb_{\delta_1} (\ep^0; \Cbb^n) $ (in the sense of \cite[Theorem IV.3.16]{Kato}). This and (ii) prove (iii) with small enough $\delta_1$. 
\end{proof}

\begin{remark} \label{r:PerRes}
Similarly to the proof of Proposition \ref{p:PerSectC}, one can obtain from (\ref{e:Te1-Te2}) and
\cite[Theorem 1.16]{Kato}  the following statement:
if $\zeta_0 \in \rho (T(\ep^0))$, then there exist positive numbers $\delta_0, \delta_1$ such that for every 
$\ep \in \Bbb_{\delta_1} (\ep^0;\Cbb^n) $  the disc $\Dbb_{\delta_0} (\zeta_0)$ is a subset of the resolvent set $ \rho (T(\ep))$.
\end{remark}

\begin{remark} \label{r:PerSectC}
The application of these perturbation results to Section \ref{s:Hypersurf}, namely, to the operator-valued function $A (\cdot,\cdot)$ restricted to
local equilibrium manifolds, requires the complexification of $\ep$ and a local holomorphic extension of $A (\cdot,\cdot)$. Let us consider this procedure in more detail.
Assume (\ref{h:HComp}), (\ref{h:HAn}), and (\ref{e:A}).
Let $(x^0;\ep^0)$ belong to the manifold (\ref{e:Meq}), 
i.e., $F (x^0;\ep^0) = 0$ and $0 \in \rho (A (x^0;\ep^0))$, where $A (x;\ep)$ is the complexification of the $x$-derivative of $D_x F (x,\ep)$. 
The real analyticity assumption (\ref{h:HAn}) on $F$ implies 
that there exists a neighborhood 
$\Ncal_{\crm}^{x^0} \times \Ncal_{\crm}^{\ep^0} \subset \Xcal_\crm \times \Cbb^n$ of $(x^0;\ep^0)$ 
where the complexified Taylor series of $F$ at $(x^0;\ep^0)$ converges 
to a certain  holomorphic  map 
$F^\loc_\crm:\Ncal_{\crm}^{x^0} \times \Ncal_{\crm}^{\ep^0}\to \Zcal_\crm$. 
The neighborhoods $\Ncal_{\crm}^{x^0}$ and $\Ncal_{\crm}^{\ep^0}$ can be chosen such that their 'real parts' $\Ncal_{\rrm}^{x^0} := \Ncal_{\crm}^{x^0} \cap \Xcal$ and $\Ncal_{\rrm}^{\ep^0} := \Ncal_{\crm}^{\ep^0} \cap \Rbb^n$ are subsets 
of the neighborhoods $\Ncal^{x^0}$ and $\Ncal^{\ep^0}$ which determine the parametrization $X^\loc$ of a local equilibrium manifold $\Gr X^\loc$ near $(x^0;\ep^0) $ (see Section \ref{s:Hypersurf}). 
At every $(x;\ep) \in \Ncal_\rrm^{x^0} \times \Ncal_\rrm^{\ep^0}$, the complex $x$-derivative $D_x F^\loc_\crm (x;\ep)$  coincides with $A (x,\ep)$. So
$A (\cdot;\cdot)$ can be extended from $\Ncal_\rrm^{x^0} \times \Ncal_\rrm^{\ep^0}$ to a holomorphic map $ D_x F^\loc_\crm : \Ncal_{\crm}^{x^0} \times \Ncal_{\crm}^{\ep^0}  \to \Lcal (\Xcal_\crm,\Zcal_\crm) $.
We put $A_\crm^\loc (\cdot;\cdot):= D_x F^\loc_\crm (\cdot;\cdot)$ and note that 
$A_\crm^\loc (X^\loc(\ep);\ep) = A (X^\loc(\ep);\ep) $ for all $\ep \in \Ncal_\rrm^{x^0}$.
The construction of the complex local equilibrium manifold $\{(X^\loc_\crm (\ep);\ep) \, : \, \ep \in \Ncal_{1,\crm}^{\ep^0} \}$ with a certain $\Cbb^n$-neighborhood $\Ncal_{0,\crm}^{\ep^0} \subset \Ncal_{\crm}^{\ep^0} $ of $\ep^0$ can be performed for the complexified map $F^\loc_\crm$ in the same way as in the real case of Section \ref{s:Hypersurf}.
The $\Lcal (\Xcal_\crm,\Zcal_\crm)$-valued function $A_\crm^\loc (\ep) := A_\crm^\loc (X^\loc(\ep);\ep) $ is holomorphic for $\ep \in \Ncal_{\crm}^{\ep^0}$.
Assume now that $A_\crm^\loc (\ep^0)$ has a nonempty resolvent set as an operator in $\Zcal_\crm$ (this assumption is valid, e.g., for all $\si(A)$-critical points $\ep$ because of the presence of isolated eigenvalues). Then $A_\crm^\loc (\ep^0)$ is closed as an operator in $\Zcal_\crm$. 
An analogue of (\ref{e:Te1-Te2}) holds for $A_\crm^\loc (\ep)$, and so \cite[Theorems IV.2.14 and IV.2.24]{Kato} imply that $A_\crm^\loc (\ep)$ are closed as operators in $\Zcal_\crm$ for all $\ep$ in a certain small $\Cbb^n$-neighborhood $V_{1,\crm}\subset \Ncal_{0,\crm}^{\ep^0}$ of $\ep_0$ (this can be also obtained from \cite[Lemma A.3.1]{L12}). Since all 
$A_\crm^\loc (\ep)$ have the same domain $\dom A_\crm^\loc (\ep) = \Xcal_\crm$, we see that $\{A_\crm^\loc (\ep)\}_{\ep \in V_{1,\crm}}$ is a holomorphic family of type ($\Acal$) in $\Zcal_\crm$.
In this way, Remark \ref{r:PerRes} is applicable to the perturbations of the resolvent set of $A_\crm^\loc (\ep^0)$, Proposition  \ref{p:PerSimpleC} is applicable to simple isolated eigenvalues, and  Proposition  \ref{p:PerSectC} is applicable to the case where 
$A_\crm^\loc (\ep^0)$ is an $\Hsec$-sectorial operator.
\end{remark}

\section{The Puiseux series proof of Lemma \ref{l:n=2Det<0}}\label{a:B}

For the convenience of the reader, a short proof of Lemma \ref{l:n=2Det<0} is given in this section.
It is based on the Weierstrass preparation theorem and \cite[Theorem 12.2]{RSIV} on the Puiseux series for roots of a complex polynomial depending on a parameter. In comparison with \cite[Section I.2.7]{VT74} this proof does not require the consideration of several related Newton diagrams.

Without loss of generality, we assume $\ep^0= (\ep^0_1, \ep^0_2) = 0 \in V \subset \Rbb^2$.
In a certain $\Cbb^2$-neighborhood $V_\crm$ of $\ep^0=0$, we consider the function 
$r_\crm (\ep_1,\ep_2) = \sum_{j,k=0}^{+\infty}  r_{j,k} \ep_1^j \ep_2^k $ 
with the real Taylor coefficients $r_{j,k}$. That is, $r_\crm$
is a complex analytic extension of the real analytic function $r$.
Under the assumptions of the lemma, $r(0)=0$, $\nabla r (0) = 0$, and the Hessian determinant $\det \Hf_r (0)$ is negative. 
In particular, $r_{0,0}=r_{1,0} = r_{0,1} = 0$.

\emph{Step 1. Let us consider the solutions $\ep_2$ of $r_\crm (\ep_1,\cdot) = 0$ in the particular case $r_{0,2} \neq 0$}.  In this case, the Weierstrass preparation theorem and \cite[Theorem 12.2]{RSIV} imply that small enough $\ep = (\ep_1,\ep_2) \neq 0$ satisfying $r_\crm (\ep) =0$ have the form $\ep = (\ep_1, \ep_2^\pm (\ep_1))$, where two Puiseux series $\ep_2^\pm (\ep_1) = \sum_{j=1}^{\infty} e^\pm_j \ep_1^{j/2} $  with certain coefficients $e^\pm_j \in \Cbb$ are  convergent in a vicinity of $\ep_1 = 0$. Moreover, the multi-set $\{ \ep_2^+ (\ep_1),\ep_2^- (\ep_1)\}$ of the zeroes of the function $r_\crm (\ep_1,\cdot)$ in a vicinity of $\ep_2 = 0$ takes the multiplicity of zeros into account and the total multiplicity of such zeroes for small enough $\ep_1$ is constant, and so equals 2, as in the case $\ep_1=0$.  

\emph{Step 2. Let us prove statement (i) of the lemma.} We consider only the case $\{ r_{2,0} = - r_{0,2} = 1, \ 
r_{1,1} = 0\}$. Every other case, where  $\det \Hf_r (0) <0 $, can be reduced to this case by a constant linear transformation of the parameter space $\Rbb^2$. Statement (i) is invariant under such transformations.

It follows from $r_{2,0} = - r_{0,2} = 1$ that $r (\ep_1,0) >0$ and $r (0,\ep_2) <0$ for
small enough nonzero $\ep_1$ and $\ep_2$. 
Hence, in every $\Rbb^2$-neighborhood of $0$, $r (\ep) = 0$ has a solution in each of the four  quadrants $\{\pm \ep_1>0, \pm \ep_2>0\}$. Comparing this with the statement of Step 1, one sees that
the two Puiseux series $\ep_2^\pm (\cdot)$ are  real locally near $0$ for real $\ep_1$ and real $\ep_2$.
This, in turn, implies that 
\begin{itemize}
\item[(a)] all the Puiseux coefficients $e^\pm_j$ are real and 
$e^\pm_j = 0$ for all odd $j$; 
\item[(b)]  $\ep_2^\pm (\cdot) $ can be chosen such that $e_2^+ \ge 0$ and $e_2^- \le 0$.
\end{itemize}
Note that statement (a) follows from the local conservation of the total multiplicity of zeroes of $r_\crm (\ep_1,\cdot)$. Indeed, if (a) is not fulfilled, the lowest order nonreal coefficient or the lowest order  nonzero coefficient with odd index produces a locally nonreal for small $\ep_1>0$ branch of the zero-locus of $r_\crm (\ep) $, which violate the local conservation of total multiplicity.

In other words, $\ep_2^\pm (\cdot) $ are real analytical functions of $\ep_1$ in an $\Rbb$-neighborhood of $0$. In order to prove the nontangentiality in the statement (i) of the lemma, it is enough to show that $e_2^\pm \neq 0$. One obtains $e_2^\pm = \pm 1$ immediately from the assumptions $r_{2,0} = - r_{0,2} = 1$ and $r_{1,1} = 0$ by plugging $\ep = (\ep_1, \ep_2^\pm (\ep_1))$ into $r_\crm (\ep) = 0$. 

\emph{Step 3. Let us prove statement (ii) of the lemma.}  
Consider first the case $r_{0,2} \neq 0$ .  Steps 1 and 2 imply that the local zero locus of $r(\cdot)$ can be represented as the union of the graphs of two real analytic functions $\ep_2^\pm (\ep_1) = \sum_{j=1}^\infty a^\pm_j \ep_1^j $ with distinct first coefficients $a^+_1 \neq a^-_1$. 
Plugging the Taylor series of $\ep_2^\pm$ into $r (\ep) = 0$,
one gets $r_{0,2} (a^\pm_1)^2 +r_{1,1} a^\pm_1 +r_{2,0} = 0 $ and, in turn,
$
a^\pm_1 = \frac{-r_{1,1} \pm \sqrt{-\det \Hf_r (0)}}{2 r_{0,2}}.
$ 
This implies (ii) in the case $r_{0,2} \neq 0$. 
The case $r_{2,0} \neq 0$ can be considered similarly. The case $r_{2,0} = r_{0,2}=0$ can be reduced to $r_{2,0} = -r_{0,2}=1$, $r_{1,1} = 0$ by a linear transformation.  This completes the proof.

\end{document}